\documentclass[a4paper,11pt,headsepline]{article}
\usepackage[english]{babel}
\usepackage[T1]{fontenc}
\usepackage[utf8]{inputenc}
\usepackage[autostyle,german=quotes]{csquotes}

\usepackage{setspace}
\usepackage[left=3.5cm,right=2.5cm,top=1cm,bottom=3cm,includeheadfoot]{geometry}

\usepackage{graphics}
\usepackage{graphicx}
\usepackage{color}
\usepackage{subfigure} 
\usepackage{caption}
\usepackage{pgf}
\usepackage{tikz}

\usepackage{amsmath}
\usepackage{shadethm}
\usepackage{amsfonts}
\usepackage[mathscr]{eucal}
\usepackage{dsfont}
\usepackage{amssymb}
\usepackage{amsthm}
\usepackage{nicefrac}

\usepackage{multicol}

\usepackage{ifthen}
\usepackage{calc}
\usepackage{nomencl}

\usepackage{scrpage2}

\theoremstyle{plain}
\newtheorem{thm}{Theorem}[section]
\newtheorem{lem}[thm]{Lemma}
\newtheorem{prop}[thm]{Proposition}
\newtheorem{cor}[thm]{Corollary}

\theoremstyle{definition}
\newtheorem{defn}[thm]{Definition}

\newtheorem{exm}[thm]{Example}

\newtheoremstyle{remark}
{15pt}
{1em}
{}
{}
{\bf}
{:}
{ }
{}

\theoremstyle{remark}
\newtheorem{rem}[thm]{Remark}

\newtheoremstyle{not}
{15pt}
{5pt}
{}
{}
{\bf}
{:\hspace{0.05cm}}
{ }
{}

\theoremstyle{not}
\newtheorem*{notation}{Notation}

\newtheoremstyle{proof}
{}
{2pt}
{}
{}
{\it}
{:}
{ }
{}

\theoremstyle{proof}

\renewcommand{\O}{\mathcal{O}}
\newcommand{\Os}{\tilde{\O}}
\newcommand{\Uad}{U^{ad}}

\newcommand{\On}{\Omega_{n}}

\newcommand{\Oext}{\Omega^{ext}}

\newcommand{\RO}{\partial \Omega}

\newcommand{\se}{\sigma}

\newcommand{\Div}{\text{div}}

\newcommand{\Umg}[2]{B_{#1}(#2)}

\newcommand{\cupdot}{\ensuremath{\mathaccent\cdot\cup}}

\newcommand{\R}[1]{\mathbb{R}^{#1}}
\newcommand{\N}[1]{\mathbb{N}^{#1}}
\newcommand{\C}[2]{C^{#1}(#2)}
\newcommand{\Ck}[3]{[C^{#1}(#2)]^{#3}}		

\newcommand{\D}[2]{\mathscr{D}^{#1}(#2)}

\newcommand{\G}{\mathcal{G}}

\newcommand{\seq}[1]{(#1)_{n\in \mathbb{N}}}
\newcommand{\subseq}[1]{(#1)_{k \in \mathbb{N}}}

\newcommand{\Norm}[2]{\Vert #1 \Vert_{#2}} 

\newcommand{\abl}[2]{D^{#1}#2}

\newcommand{\dist}{\text{dist}}
\newcommand{\diam}{\text{diam}}

\newcommand{\A}{\mathcal{A}}
\newcommand{\T}{\mathcal{T}}

\begin{document}

\title{\sc Optimal Reliability for Components under Thermomechanical Cyclic Loading}

\author{{\sc Laura Bittner and Hanno Gottschalk}\\{\small  School of
Mathematics and Science}\\ \small{Bergische Universit\"at
Wuppertal}\\
{\small \tt laura.bittner@math.uni-wuppertal.de}\\
{\small \tt hanno.gottschalk@uni-wuppertal.de}}

\maketitle

\vspace{.3cm}

\begin{abstract}
We consider the existence of optimal shapes in the context of the thermomechanical system of partial differential equations (PDE) using the recent approach based on elliptic regularity theory \cite{Hanno} \cite{Agm59,Agm64,GilbTrud}. We give an extended and improved definition of the set of admissible shapes based on a class of sufficiently differentiable deformation maps applied to a baseline shape. The obtained set of admissible shapes again allows one to prove a uniform Schauder estimate for the elasticity PDE. In order to deal  with thermal stress, a related uniform Schauder estimate is also given for the heat equation. Special emphasis is put on Robin boundary conditions, which are motivated from convective heat transfer. It is shown that these thermal Schauder estimates can serve as an input to the Schauder estimates for the elasticity equation \cite{Hanno}. This is needed to prove the compactness of the (suitably extended) solutions of the entire PDE system in some state space that carries a $C^2$-H\"older topology for the temperature field and a $C^3$-H\"older topology for the displacement. From this one obtains he property of graph compactness, which is the essential tool in an proof of the existence of optimal shapes. Due to the topologies employed, the method works for objective functionals that depend on the displacement and its derivatives up to third order and on the temperature field and its derivatives up to second order. This general result in shape optimization is then applied to the problem of optimal reliability, i.e. the problem of finding shapes that have minimal failure probability under cyclic thermomechanical loading.       
\end{abstract}




\section{\label{sec:Intro}Introduction}

Objective functionals that are motivated by failure probabilities of mechanical components that are subject to cyclic mechanical loading have been introduced in \cite{BGS,Hanno,Schmitz} to the field of shape optimization \cite{DelfZol11,ShapeOpt,SokZol92}. Here, failure times are modelled by spatio-temporal Poisson Point Processes (PPP) and their first occurrence times.  In this paper, we  take enhanced material damage into account that occurs at elevated temperature. This is due to thermal stresses and also due to reduced durability of materials at higher temperatures. The assumed design objective is to choose the shape of a component in a set of admissible shapes such that the failure probability after a given number of load cycles is minimal. 

The reliability assessment of cooled components, e.g. in gas turbines or vessels, leads to a set of multi-physical partial differential equations that is known as the the thermo-mechanical equation \cite{HetEsl09}.  In this paper, we investigate a shape optimization problems with the thermo-mechanical system of PDEs as state equation. We consider fairly singular set of objective functionals that are motivated by the probability of failure under cyclic themo-mechanical loading, as it is the case for low cycle fatigue (LCF). This set of objective functionals introduce temperature dependence to the objective functionals in \cite{Hanno,SBKRSG,SSGBRK}. We show that, under suitable regularity assumptions on the admissible shapes and the boundary conditions, there exist of shapes with minimal failure probability.

The boundary conditions to the thermal equation in this paper are of Robin type. This corresponds to convective heat transfer at the boundary of the component, which is most frequently used in engineering applications. For the coupling of thermal stress to the mechanical equation, we follow a partially coupled approach, see e.g. \cite{HetEsl09}. When writing the coupled system in strong form, the gradient of the temperature field becomes a volume force density for the elasticity PDE and the temperature difference to a baseline temperature, where the component is in its original stress free state, becomes a surface load density in the elasticity equation. This re-defines the right hand side of the elasticity equation. 

In this situation, uniform regularity estimates for the temperature fields serve two purposes: On the one hand, the temperature field itself is part of the solution of the state equation in the sense of shape optimization with PDE constraints. We thus need some compactness for this component of the state space. Secondly, regularity assumptions on volume forces and surface loads are crucial input for the regularity estimates for the elasticity equation. Here uniform bounds on $C^1$-H\"older norms for the volume force densities and $C^2$-H\"older for the surface load densities are required for the uniform Schauder estimates in \cite{Agm59,Agm64,GilbTrud,Hanno} . In this paper we prove that such estimates in fact hold true. Having established suitable uniform Schauder estimates on both components of the state space, we can now proceed to prove graph compactness in the sense of shape optimization using the Arzela-Ascoli theorem for H\"older spaces \cite[Section 8]{Alt}. This implies existence of optimal shapes in the class of admissible shapes for all objective functionals that are continuous with respect to the state space topology, which is a direct sum topology for the temperature and displacement field as specified above.

Proofs for the existence of optimal shapes are not new in shape optimization, see e.g. \cite{BGS,BucBut05,Chen75,DelfZol11,Eppler07,ShapeOpt}. However, the class of objective functionals that arise from component reliability are too singular to deal with them in the framework of weak solutions, as done in the cited references. Like in \cite{Hanno,Schmitz}, we therefore need to extend the general strategy of shape optimization in order to be able to deal with the objective functionals that arise from component reliability.  

The objective of this paper is to give a mathematical existence proof of a shape optimization problem in a context which is as close to a real design problem as possible -- taking high temperature design in gas turbine engineering as a model. Although this intention can only be realized partially, we show that the machinery of elliptic regularity theory and the general theory of shape optimization \cite{ShapeOpt} is powerful enough, to deal with certain non oversimplified problems in a mathematically rigorous way.          
       
The paper is organized as follows: In Section \ref{sec:ProbFail} we review crack initiation processes and their relation to shape optimization following essentially \cite{Hanno}. However, more and different notions of optimal reliability are introduced and compared. It is shown that in the case of Weibull models, all these different notions coincide, which allows to prove existence of shapes with optimal reliability in a stronger sense as given in \cite{Hanno}. It is shown how the problem of optimal reliability is related to shape optimization problems. We also extend the crack initiation model for LCF to the case of non constant temperature fields using an approach based on Arrhenius' law. 

 In Section \ref{sec:PDE} we review the themo-mechanical PDE as the state equation to our problem.  Section \ref{sec:Basics} gives some background from the abstract theory of shape optimization following \cite{ShapeOpt}. In Section \ref{sec:Domains} we present a new and enlarged set of admissible domains based on $C^{k,\alpha}$ deformations of a baseline shape. Compactness results on the set of admissible shapes are given. 

Section \ref{sec:UnifSchauder} contains the uniform Schauder  estimates for the mechanical, the thermal and the thermo-mechanical state equations. Here special emphasis is laid on realistic convective (Robin) boundary conditions for the heat equation. The exposition is based on \cite{GilbTrud} and proves the uniformity of the regularity estimates in this reference with respect to our set of admissible shapes. 

In the following Section \ref{sec:OpShapes} the uniform Schauder estimates are used to prove graph compactness of a large class of rather singular shape optimization problems of local type, which includes those derived from optimal reliability. 

Finally we give a summary and outlook in Section \ref{sec:SO}. An appendix collects some technical results from the literature for the convenience of the reader.  
\section{\label{sec:ProbFail}Failure Probabilities and Objective Functionals in Shape Optimization}

In this section, we derive failure probabilities of mechanical components under thermomechanical loading and give some specific models that are motivated, in a wide sense, from gas turbine design or vessel design.

\subsection{\label{subsec:SFTM} Stochastic Failure Time Models and Point Processes}

Let $\Omega\subseteq \R{3}$ be some bounded domain with boundary $\partial\Omega$. By $\bar \Omega$ we denote the closure of $\Omega$ in $\R{3}$. In this article, $\Omega$ stands for a region in $\R{3}$ that is filled with some material and that symbolizes a mechanical device. Initially this component has a given reference temperature $T_0$ and there are no loads that could deform $\Omega$. Under operation, the device $\Omega$ undergoes a deformation which is defined by the displacement field $u=u(\Omega):\bar \Omega\to \R{3}$. The displacement is caused by mechanical volume loads $f(\Omega):\Omega\to\R{3}$ (e.g. gravity or centrifugal loads) and surface loads $g(\Omega):\partial \Omega\to\R{3}$ (e.g. gas pressure). The device might be clamped at some locations, i.e. $u(\Omega)(x)=0$, $x\in \partial\Omega_D\subseteq \Omega$. 

In the course of the load cycle, some heating and cooling might take place, such that $u(\Omega)$ also depends on a temperature distribution $T(\Omega):\bar \Omega\to \R{}$. The temperature-dependence of $u(\Omega)$ is  due to thermal expansion. The temperature distribution inside $\Omega$ in turn will depend on external temperatures $T_e(\Omega):\partial\Omega\to\R{}$ and the heat transfer mechanism, that is generally modelled by a heat transfer coefficient $k(\Omega):\partial\Omega\to \R{}$. 

Usually, the thermal and mechanical loads lead to a deterioration of the material, also known as fatigue. This degeneration will result in the formation of cracks that in the end will destroy the component. The prediction of the number of load cycles (or time) that will be passed safely before crack formation takes place is difficult. Hence, it is more realistic to set up probabilistic models for crack formation. 

Let $\bar\Omega$ be the closure of $\Omega\subseteq \mathbb{R}^3$ and let $\mathcal{C}=\mathbb{R}_+\times \bar\Omega$ be the configuration space for crack initiation. I.e.\ each crack initiation on the initially crack free component is identified with a location $x\in\bar\Omega$ and a time $t\in\mathbb{R}_+$.

Let $\mathcal{R}(\mathcal{C})$ be the space of Radon measures on $\mathcal{C}$ and $\mathcal{R}_c(\mathcal{C})$ be the space of Radon counting measures. This kind of measures maps measurable sets in the Borel $\sigma$-Algebra $\mathcal{B}(\mathcal{C})$ to $\mathbb{N}_0\cup\{\infty\}$. We equip the space of Radon (counting) measures with the weak-*-topology that is generated by the mappings $\mathcal{R}(\mathcal{C})\ni \gamma\to \int_\mathcal{C}hd\gamma$, $h\in C_c(\mathcal{C})$. The associated Borel $\sigma$-algebra is denoted by $\mathcal{B}(\mathcal{R})$ and $\mathcal{B}(\mathcal{R}_c)$, respectively. $C_c({\cal C})$ are the continuous functions on ${\cal C}$ with compact support.

A Radon counting measure $\gamma$ is called simple, if for every bounded set $A\in \mathcal{B}(\mathcal{C})$ there exists $n<\infty$ and $c_j\in A$, $j=1,\ldots,n$ all distinct such that $\gamma\restriction_A=\sum_{j=1}^n\delta_{c_j}$.

\begin{defn}[Crack Initiation Process]
Let $(\Xi,{\cal A},P)$ be a probability space and $\gamma :(\Xi,{\cal A},P)\to (\mathcal{R}(\mathcal{C}),\mathcal{B}(\mathcal{R}_c))$ be measurable. Then $\gamma$ is called a point process.
\begin{itemize}
\item[(i)] If a point process $\gamma$ is almost surely simple and non-atomic, i.e.\ $\gamma(\{c\})=0$ holds $P$ a.s. $\forall c\in\mathcal{C}$, then we call $\gamma$ a crack initiation process on $\bar \Omega$.
\item[(ii)] For a crack initiation process $\gamma$ we define $\tau =\tau(\gamma)=\inf\{t\geq 0:\gamma([t,\infty)\times \bar\Omega)>0\}$, the fist failure time associated with $\gamma$. Note that $\tau:(\Xi,{\cal A},P)\to ([0,\infty],{\cal B})$ is a random variable, where ${\cal B}$ is the extended Borel $\sigma$-Algebra.
\end{itemize}
\end{defn}

The notion of a crack initiation process reflects the stochastic nature of crack formation that has been widely studied in the materials science literature, see e.g. \cite{Werkstoffe}. The process is chosen to be simple since no two cracks can initiate at the same location and the same time (in that case they would form one crack). Non atomicness is motivated by the fact that in non deterministic crack formation processes there should be no point on the component, where the probability that a crack originates exactly there, is larger than zero. 

We now investigate the situation in which cracks have not yet grown to a size where they can influence the macroscopic stress field. In such a situation, it is reasonable to think of the various crack initiations to be independent. This simplifying assumption is justified in the study of first failure times, as we do here.

\begin{defn}[Independent Increments and Poisson Point Process]
\label{def:PPP}
Let $\gamma :(\Xi,{\cal A},P)\to (\mathcal{R}(\mathcal{C}),\mathcal{B}(\mathcal{R}_c))$ be a point process on $\mathcal{C}=[0,\infty)\times\bar \Omega$. 
\begin{itemize}
\item[(i)] The point process $\gamma$ has independent increments, if for $C_1,\ldots, C_n\in {\cal B}(\mathcal{C})$ mutually disjoint we have that $\gamma(C_1),\ldots,\gamma(C_n)$ are independent random variables.
\item[(ii)] The point process $\gamma$ is a Poisson Point Process (PPP) if $\exists\,\rho\in {\cal R}(\mathcal{C})$ such that $\forall C\in{\cal B}(\mathcal{C})$, $\gamma(C)$ is Poisson distributed with intensity $\rho(C)$, i.e. $P(\gamma(C)=n)=e^{-\rho(C)}\,\nicefrac{\rho(C)^n}{n!}.$
\end{itemize} 
\end{defn} 

For a crack initiation process the property of having independent increments is equivalent to being a PPP, see \cite{Wahrsch}. Thus, if we accept the assumption of independent increments, we only have to model the intensity measure $\rho$ as a function of the stress and the temperature state on $\bar\Omega$.   

Let ${\cal O}$ be some set of admissible domains $\Omega\subseteq \R{3}$ with associated temperature fields $T=T(\Omega)$ and displacement field $u=u(\Omega)$. It is then natural to model the local crack initiation intensity as a function of time and local values of the temperature and the displacement along with their derivatives. Here we restrict ourselves to derivatives up to third order in $u$ and up to second order in $T$:
\begin{align}
\label{eqa:locFailure}
\begin{array}{rcl}
\rho(\Omega,C)&=&\int_{C\cap (\R{}\times \Omega)}\varrho_{vol}(t,x,T,\nabla T,u,\nabla u,\nabla^2 u,\nabla^3u)\,dt\,dx\\
&+&\int_{C\cap (\R{}\times \partial \Omega)}\varrho_{sur}(t,x,T,\nabla T,u,\nabla u,\nabla^2 u,\nabla^3u)\,dt\,dA.\end{array}
\end{align}
Here $dA$ stands for the surface measure on $\partial\Omega$ and $\varrho_{vol/sur}$ are some non negative functions that depend on the physics of the crack formation. The function $\varrho_{vol}$ represents volume driven failure mechanisms, like e-g-creep, whereas $\varrho_{sur}$ models surface driven crack formation, like e.g. low cycle fatigue (LCF). We will be more specific in Subsection \ref{subsec:LCF} below. Here we derive some immediate consequences on the probability distribution of the first failure time:

\begin{lem}
\label{lem:Prob}
Let $\gamma=\gamma(\Omega)$ be the PPP associated with (\ref{eqa:locFailure}) and $\tau=\tau(\gamma)$ the associated first failure time. Let $C_t=[0,t]\times \bar \Omega$ and $H(\Omega,t)=\rho(\Omega,C_t)$. Then $H(\Omega,t)$ is the cumulative hazard rate of the random variable $\tau$, i.e. we have for the cumulative distribution function $F_\tau(t)$
\begin{equation}
 \label{eqa:CDFFFT}
F_\tau(t)=P(\tau\leq t)=1-e^ {-H(\Omega,t)},~~t\in\R{}. 
\end{equation}
\end{lem}

\begin{proof}
Note that $P(\tau> t)=P(\gamma(C_t)=0)=e^ {-\rho(\Omega,C_t)}$ and go over to the complementary probabilities.
\end{proof}

\subsection{\label{subsec:OR}Optimal Reliability Problems}

The problem of optimal reliability can be formulated on several levels. Every choice of a component shape $\Omega$ in the design process induces the probability distribution (\ref{eqa:CDFFFT}). Hence, it is not obvious how to compare the distribution of $\tau=\tau(\Omega)$ with that of $\tau'=\tau(\Omega')$. The following definition gives a number of alternatives:

\begin{defn}[Different Notions of Reliability]
\label{def:Reliability}
Let $\tau$ and $\tau'$ be two first failure times associated via  Definition \ref{def:PPP} and (\ref{eqa:locFailure}) to the design alternatives $\Omega,\Omega'\subseteq \R{3}$. 
\begin{itemize}
\item[(i)] The design $\Omega$ is more or equally reliable  than $\Omega'$ at fixed time $t\in\R{}_+$, if the probability of failure is less for $\Omega$ than for $\Omega'$, hence $F_\tau(t)\leq F_{\tau'}(t)$.
\item[(ii)] The design $\Omega$ is more or equally reliable than $\Omega'$ in first stochastic order, if it is more or equally reliable in the sense (i) for any time $t\in \R{}_+$.
\item[(iii)] Suppose that $\tau$ and $\tau'$ are continuously distributed. Then $\Omega$ is more reliable than $\Omega'$ in the sense of instantaneous hazard, if $h_\tau(t)\leq h_{\tau'}(t)$ holds $\forall t\geq 0$. Here $h_\tau(t)=f_\tau(t)/(1-F_\tau(t))$ is the Hazard rate and $f_\tau$ the density function of $\tau$.
\end{itemize}  
\end{defn} 

Clearly, each of this notions of 
reliability gives rise to an optimal reliability problem:
Solutions of the optimal reliability problem \ref{def:OptReliability} with respect to first order stochastic dominance are interesting because a product with optimal design $\Omega^*$ serves the customer more reliably until any time of its life cycle -- design to life is excluded.
The concept of higher reliability in terms of instantaneous hazard enhances this notion: One design is not only more reliable for any time span $[0,t]$, but this also holds at each instance in time. Since $F_\tau(t)=1-e^{-\int_0^th_\tau(s) ds}$ \cite{EscobarMeeker}, $h_\tau(t)\leq h_{\tau'}(t)~\forall t\in\R{}_+$ implies $F_\tau(t)\leq F_{\tau'}(t)$ $\forall t\in\R{}_+$ and thus the concept of higher reliability in instantaneous hazard is more  restrictive than the higher reliability in first stochastic order.

\begin{defn}[Optimal Reliability Problem]
\label{def:OptReliability}
Let ${\cal O}$ be some set of admissible domains $\Omega\subseteq\R{3}$. 
Then, $\Omega^*\in{\cal O}$ solves the problem of optimal reliability according to (i), (ii) or (iii) of Definition \ref{def:Reliability}, if it is more or equally reliable than any other design $\Omega\in{\cal O}$ in the given sense.
\end{defn}

Let ${\cal F}_{vol/sur}(t,\cdot)=\int_0^t \varrho_{vol/sur}(s,\cdot)\, ds$ and
\begin{equation}
\label{eqa:ObjFunct1}
{\cal J}_t(\Omega,u,T)={\cal J}_{vol,t}(\Omega,u,T)+{\cal J}_{sur,t}(\Omega,u,T)
\end{equation} 
with
\begin{align}
\label{eqa:ObjFunct2}
\begin{array}{rcl}
{\cal J}_{vol,t}(\Omega,u,T)&=&\int_{\Omega}{\cal F}_{vol}(t,x,T,\nabla T,u,\nabla u,\nabla^2 u,\nabla^3u)\,dx\\
{\cal J}_{sur,t}(\Omega,u,T)&=&\int_{\partial \Omega}{\cal F}_{sur}(t,x,T,\nabla T,u,\nabla u,\nabla^2 u,\nabla^3u)\,dA\end{array}.
\end{align}  
Then the optimal reliability problems Definition \ref{def:OptReliability} can be translated to the following shape optimization problems:

\begin{lem}
\label{lem:OpRelSO}
Let the crack initiation process $\gamma=\gamma(\Omega)$ for some $\Omega\in{\cal O}$ be a PPP with intensity measure (\ref{eqa:locFailure}). 
\begin{itemize}
\item[(i)] A shape $\Omega^*\in{\cal O}$ solves the optimal reliability problem (i) at fixed time $t\in \R{}_+$ if and only if 
\begin{equation}
\label{eqa:SO}
{\cal J}_t(\Omega^*,u,T)\leq {\cal J}_t(\Omega,u,T)~~\forall\, \Omega\in{\cal O}.
\end{equation}
\item[(ii)] Furthermore a shape $\Omega^ *$ solves the optimal reliability problem \ref{def:OptReliability} in the sense of first order stochastic dominance \ref{def:Reliability} (ii), if and only if $\Omega^*$ solves (\ref{eqa:SO}) for all $t\in\R{}_+$.   
\item[(iii)] Finally, a shape $\Omega^ *$ also solves the optimal reliability problem in the sense of instantaneous hazard, if an only if
\begin{equation}
\label{eqa:OptHaz}
\frac{d{\cal J}_t(\Omega^*,u,T)}{dt}\leq \frac{d{\cal J}_t(\Omega,u,T)}{dt}~~\forall t\in\R{}_+,~\Omega\in{\cal O}.
\end{equation} 
\end{itemize}
\end{lem}  
\begin{proof}
This follows from $P(\tau\leq t)=1-e^{-{\cal J}_{vol,t}(\Omega,u,T)}$, see Lemma \ref{lem:Prob} and equations (\ref{eqa:locFailure}) and (\ref{eqa:ObjFunct1}--\ref{eqa:ObjFunct2}). Thus, $h_\tau(t)= \frac{d{\cal J}_t(\Omega^*,u,T)}{dt}$.

\end{proof}

\begin{rem}
	Another perspective to the problem of optimal reliability is given by the notion of acceptability functionals: 
	
	Let $\mathscr{A}(\tau)$ be an acceptability functional in the sense of \cite{PR}. Common choices include the life expectation $\mathscr{A}(\tau)=\mathbb{E}[\tau]$ or risk adjusted versions of it, e.g. $\mathscr{A}(\tau)=\mathbb{E}[\tau]-\delta {\rm Var}[\tau]$ for some $\delta>0$. Furthermore, for $\phi:\R{}\to\R{}$ measurable, $\mathscr{A}_\phi(\tau)=\mathbb{E}[\phi(\tau)]$ also defines an acceptability functional, provided the $\phi(\tau)$ is in $L^ 1(\Xi,P)$. Then, the design $\Omega$ is more or equally reliable than $\Omega'$ with respect to $\mathscr{A}$, if ${\cal A}_\phi(\tau)\geq{\cal A}_\phi(\tau')$.   
	
	Hence, a shape $\Omega^ *$ solves the optimal reliability problem in the sense of acceptability for all $\mathscr{A}_\phi$ with increasing $\phi$, if and only if it solves (\ref{eqa:SO}) $\forall t\in\R{}_+$, see the equivalent formulations of first order stochastic dominance in \cite[Theorem 1.13 (i)]{PR}.
	
	Obviously the ranking of failure or survival probabilities at a given warranty time or service interval in \ref{def:Reliability} (i) is a special case of the general notion of acceptability with $\phi=1_{\{\tau >t\}}$.
\end{rem}

Interestingly, there exists a special situation, when the solution of the optimal reliability problem in the sense of instantaneous hazard can be obtained by finding at least one solution to the related problem at a fixed time $t$.

\begin{defn}[Local Weibull Model]
\label{def:LocWeibull}
Let $m>0$ be a Weibull shape parameter and 
\begin{equation}
\label{eqa:locWei}
\varrho_{vol/sur}(t,\cdot)=\frac{m}{N_{vol/sur}(\cdot)}\left(\frac{t}{N_{vol/sur}(\cdot)}\right)^{m-1}
\end{equation}
for some functions $N_{vol/sur}(\cdot)=N_{vol/sur}(x,T,\nabla T,u,\nabla u,\nabla^ 2u,\nabla^ 3u)$ with values in $[0,\infty]$. The associated crack initiation processes are called local Weibull models. 

Note that the convention $\frac{1}{\infty}=0$ is used here and one of $N_{vol/sur}(\cdot)$ could be identically infinite.  
\end{defn} 

Since $ N_{vol/sur} $ can be interpreted as the number of load cycles passed until a crack forms, $ N_{vol}= 0$ or $ N_{sur}=0 $ means that the rupture originates either in the volume or at the surface of the mechanical device. An example for the derivation of such an functional $ N_{sur} $ will be presented in the next section.\\ 

We recall that a random variable $\tau$ is Weibull distributed with scale $N$ and shape $m$, $\tau \sim {\rm Wei}(N,m)$, if $F_\tau(t)=1-e^ {-\left(\frac{t}{N}\right)^ m}$. 

\begin{lem}
\label{lem:Wei}
Let $\gamma=\gamma(\Omega)$ be the PPP from a local Weibull model. Then the first failure time  $\tau=\tau(\gamma)$ is Weibull distributed, $\tau\sim {\rm Wei}(N,m)$, with $N=N(\Omega)$ given by
\begin{align}
\label{eqa:scaleWei}
\begin{array}{rcl}
N&=&\left(\int_\Omega \left(\frac{1}{N_{vol}(x,T,\nabla T,u,\nabla u,\nabla^ 2u,\nabla^ 3u)}\right)^ mdx\right.\\
&+&\left.\int_{\partial\Omega}\left(\frac{1}{N_{sur}(x,T,\nabla T,u,\nabla u,\nabla^ 2u,\nabla^ 3u)}\right)^m dA\right)^{-\frac{1}{m}}.
\end{array}
\end{align}
\end{lem} 
\begin{proof}
Insert (\ref{eqa:locWei}) into (\ref{eqa:locFailure}) and (\ref{eqa:CDFFFT}). The integral in time now  up to some maximal time $t$ can be easily solved as it factors out.  We obtain $F_\tau(t)=1-e^{-t^mN^{-m}}$ with $N$ given by (\ref{eqa:scaleWei}). 
\end{proof}

In the context of a local Weibull model, the problem of optimal reliability in first order stochastic dominance can be reduced to a simple shape optimization problem, as the following proposition shows:

\begin{prop}
\label{prop:OpRelFOSD}
Let $\gamma=\gamma(\Omega)$ be the crack initiation process associated with a local Weibull model where $m\geq 1$ and let ${\cal O}$ be the set of admissible shapes. Then, 
\begin{itemize}
\item[(i)] $\Omega^*\in{\cal O}$ is a solution to the optimal reliability problem \ref{def:OptReliability} (iii) if and only if it solves the optimal reliability problem \ref{def:OptReliability} (i) at a given time $t$.
\item[(ii)]    $\Omega^*\in{\cal O}$ is a solution to the optimal reliability problem \ref{def:OptReliability} (i) if and only if 
\begin{equation}
\label{eqa:SOWei}
{\cal J}(\Omega^*,u,T)\leq {\cal J}(\Omega,u,T)~~\forall \Omega\in{\cal O}.
\end{equation}
for
\begin{align}
\label{eqa:scaleWei2}
\begin{array}{rcl}
{\cal J}(\Omega^*,u,T)&=&\int_\Omega \left(\frac{1}{N_{vol}(x,T,\nabla T,u,\nabla u,\nabla^ 2u,\nabla^ 3u)}\right)^ mdx\\
&+&\int_{\partial\Omega}\left(\frac{1}{N_{sur}(x,T,\nabla T,u,\nabla u,\nabla^ 2u,\nabla^ 3u)}\right)^m dA.
\end{array}
\end{align}
\end{itemize} 
\end{prop}
\begin{proof}
(i) Let $t\in\R{}_+$ be fixed. If $\Omega^ *$ solves the optimal reliability problem \ref{def:OptReliability} (iii) with respect to that time, we have $F_{\tau(\Omega^*)}(t)\leq F_{\tau(\Omega)}(t)$ $\forall\, \Omega\in{\cal O}$ and thus 
$$
1-e^{-t^mN(\Omega^*)^{-m}}\leq 1-e^{-t^mN(\Omega)^{-m}} ~\Leftrightarrow~ N(\Omega^*)\geq N(\Omega)~~\forall \Omega\in{\cal O}.$$
But then the hazard rates fulfill for $m\geq 1$
\begin{equation}
h_{\tau(\Omega^*)}(t)=\frac{m}{N(\Omega^*)}\left(\frac{t}{N(\Omega^*)}\right)^{m-1}\leq \frac{m}{N(\Omega)}\left(\frac{t}{N(\Omega)}\right)^{m-1}=h_{\tau(\Omega)}(t)~,~~\forall t\in\R{}_+.
\end{equation}

(ii) Combine (i) for $t=1$, Lemma \ref{lem:Wei}, Def.\ \ref{def:LocWeibull} and  Lemma \ref{lem:OpRelSO} (i).
\end{proof}
\subsection{\label{subsec:LCF}The Example of Low Cycle Fatigue} 

In this subsection we extend the local Weibull model for Low Cycle Fatigue (LCF) for the purely mechanical load case in \cite{Hanno} to the case of thermo-mechanical loading. It is well known \cite{Werkstoffe} that repeated loading of a mechanical component ultimately leads to failure, even if the single loads are well below the ultimate tensile strength of the material. This degradation of  strength is known as fatigue. LCF is a damage mechanism that is best understood for polycrystalline metal. In LCF, shear stress acting on atomic layers with the densest packing leads to the intragranular displacement of one dimensional lattice defects. When these defects reach the surface of the component, intrusions and extrusions form, see Figure \ref{fig:Damage} (a). This leads to stress concentration at the tip of the intrusion, from which a crack originates. Perculation of the initial intragranular crack over grain boundaries then causes macroscopic cracks, \cite{Werkstoffe}.  This is why LCF is a stress and surface driven failure mechanism, see Fig \ref{fig:Damage} (b).

\begin{figure}
\includegraphics[scale=.23]{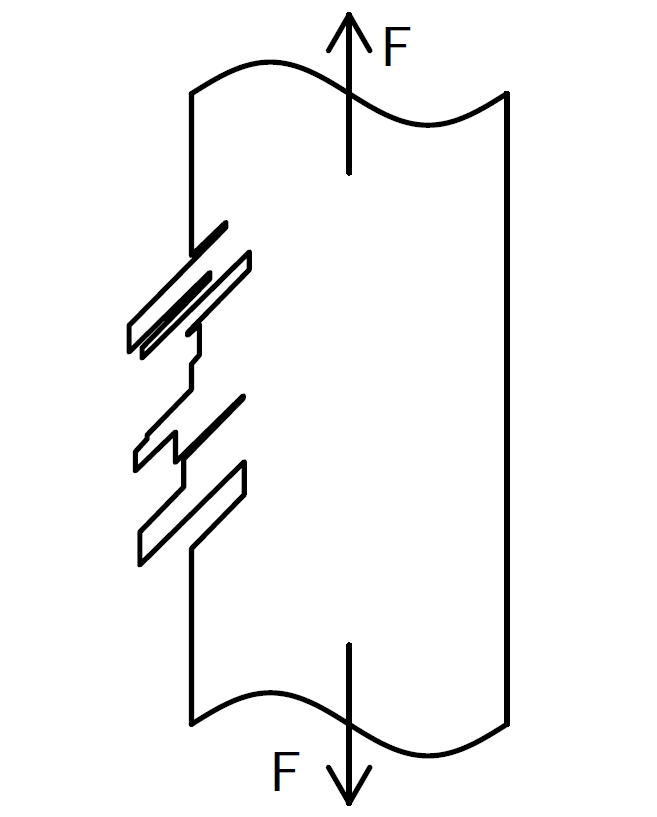}~\includegraphics[width=0.6\textwidth]{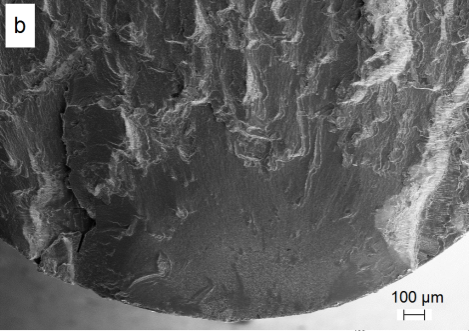}

\caption{(a) Intrusions and extrusions at the surface forming under cyclic application of the force $F$. (b) Crack initation at the lower boundary of a specimen cracked during a cyclic life test for the Ni-based superalloy RENE80. }
\label{fig:Damage}
\end{figure}

Let  $\sigma=\sigma(\nabla u,T):\Omega\to \R{3\times 3}$ be the stress field associated with the displacement field $u$ (the first derivatives thereof, in particular) and the temperature field $T$  via a material equation. Here we suppress the  $\Omega$ dependence for notational simplicity. For the example of linear thermo-elasticity we refer to the following Section \ref{sec:PDE} Eq.\ (\ref{Def: stress}).  

By $ \sigma'=\sigma-\frac{1}{3}{\rm tr}(\sigma)I$ we denote the trace free part of the stress field, where $I$ stands for the identity matrix on $\R{3}$. We shortly recall the steps that lead to the calculation of the approximate number of load cycles to crack initiation $N_{sur}$ for the case of cyclic, purely mechanical loading
\begin{enumerate}
\item Define the amplitude comparison stress as the von Mises stress associated with $ \sigma$, i.e. $ \sigma_v=\sqrt{\frac{3}{2} \sigma': \sigma'}$ and define the amplitude stress as $\epsilon_a=\sigma_v/2$.

\item If $ \sigma$ is obtained from a linear thermo-elasticity, convert  the amplitude stress $ \sigma_a$ to elastic-plastic amplitude stress, e.g. via the Neuber relation
\begin{equation}
\label{eqa:SD}
\frac{\sigma_a^2}{E}=\frac{( \sigma_a^{el-pl})^2}{E}+ \sigma_a^{el-pl}\left(\frac{\sigma_a^{el-pl}}{K}\right)^{1/n'}.
\end{equation}
Otherwise, i.e. if $\sigma$ is obtained from an thermo-elastoplastic problem, set $ \sigma_a^{el-pl}=\sigma_a$. In equation (\ref{eqa:SD}) $E$ stands for Young's modulus, $K$ for the hardening constant and $n'$ for the hardening exponent.
\item Convert the elastic-plastic comparison stress amplitude to the elastic-plastic strain amplitude $\varepsilon_a^{el-pl}$ via the Ramberg-Osgood relation:
\begin{equation}
\label{eqa:RO}
\varepsilon_a^{el-pl}=\frac{ \sigma_a^{el-pl}}{E}+\left(\frac{\sigma_a^{el-pl}}{K}\right)^{1/n'}.
\end{equation}
\item Solve the Coffin-Manson-Basquin equation for $\mathcal{N}_{sur}$,
\begin{equation}
\label{eqa:CMB}
\varepsilon_a^{el-pl}=\frac{\sigma_f'}{E}(2\mathcal{N}_{sur})^b+\varepsilon_f'(2\mathcal{N}_{sur})^c.
\end{equation}
Here $\sigma_f',\varepsilon_f'>0$ and $b,c<0$ are material constants.
\end{enumerate}

Let us now turn to the case, when, in addition to a mechanical load, a temperature change from the baseline temperature $T_0$ to the temperature field $T$ takes place. It is then usually assumed that the durability changes with temperature. Dissemination of displacements through the crystal is facilitated by the thermal excitation of the atomic oscillations in the lattice. Therefore, LCF life to crack initiation usually decreases with temperature. Here, we take a simplistic Arrhenius law \cite{Werkstoffe} for the temperature dependence of LCF life. This is an 'import' from creep damage modelling and not necessarily the most adequate temperature model. Alternatively one could consider temperature dependent CMB parameters $\sigma_f'(T),\varepsilon_f(T),b(T)$ and $c(T)$ modelled by continuous functions. The field of temperature models in fatigue however is vast and lies beyond the scope of this article. We thus choose 
\begin{equation}
N_{sur}(\cdot)=e^{-Q(T-T_0)}\mathcal{N}_{sur}(\cdot).
\end{equation}  
Here $Q$ plays the r\^ole of an activation energy. Using (\ref{eqa:CMB}) it is easily shown that this corresponds to $b(T)=b$ and $c(T)=c$ constant and $ \frac{\sigma_f'(T)}{E}=e^{-Qb(T-T_0)}\frac{\sigma_f'}{E}$ and $\varepsilon_f'(T)=e^{-Qc(T-T_0)}\varepsilon_f'$. 
Wrapping up these modelling steps, we obtain the following:

\begin{lem}
\label{lem:Model}
Let $N_{sur}=N_{sur}(\nabla u,T)$ be defined as above. Then $\left(\frac{1}{N_{sur}}\right)^m$ depends continuously on $\nabla u$ and $T$. 
\end{lem}

Note that models that include notch support factors \cite{Werkstoffe} also require derivatives $\nabla^2u$. Therefore second order derivatives will enter into the definition of $N_{sur}$.

We have thus shown that in realistic models, the functions ${\cal F}_{vol/sur}$ from  (\ref{eqa:ObjFunct2}) can be assumed to be continuous, as long as the temperature does not exceed a certain limit $T_m$.

\section{\label{sec:PDE}The Thermomechanic Formulation of Linear Elasticity} \label{Sec: PDE constraints}
Up to here, we have seen that volume and surface forces as well as changes in temperatures that are imposed to a device have an impact on its durability. In linear thermo-elasticity both aspects are taken into account, see \cite{HetEsl09}. \\[1ex]
As proposed in \cite{Hanno} and for reasons of simplification, we restrict ourselves to the case that the load vector fields $ f $ and $ g $ are independent of time, as well as the temperature distribution field $ T $. This means that the time $ t $ only counts the number of load cycles passed.  Then, according to \cite{ErnGuerm04} the disjoint displacement-traction problem of linear isotropic elasticity is defined by 
\begin{align}
	\left. 
		\begin{array}{ll}
			 \Div( \se(u)) +f=0   &\text{ in } \Omega \\
			 \se(u)=\lambda\Div(u)I+\mu(\nabla u+(\nabla u)^{\top}) &\text{ in } \Omega\\ 
		 	 u = 0  &\text{ auf } \partial\Omega_{D}  \\
		 	 \se(u) \cdot \nu =g &\text{ auf }\partial \Omega_{N} \label{PO}
		 \end{array} 
	\right.
\end{align}
on $ \Omega \subset \R{3} $. Here, $ \RO_{N} \cupdot \RO_{D} $ is a partition of the domain's boundary, where on $ \RO_{N} $ a force surface density $ g\restriction_{\RO_{N}} $ is imposed and $ \RO_{D} $ is clamped. Let $\nu $ be the outward normal on $ \RO $. The load vector field $ f:\Omega \to \R{3} $ corresponds to a force imposed on the volume of $ \Omega $ and every solution $ u:\overline{\Omega}\to\R{3} $ is called displacement field on $ \Omega $. $ I $ denotes the identity on $ \R{3} $. In addition we assume that the Lamé-coefficients $\lambda,\mu>0$ are constants. For the computation of approximative numerical solutions a finite element approach can be used, confer \cite{ErnGuerm04} and \cite{HetEsl09}.

Temperature gradients, as for example occurring in materials that are heated and cooled at the same time, encourage abrasion as already explained in the preceded section. Therefore, they have to be taken into account when dealing with durability of devices that are run under changing temperatures. For implementation we use a combination of the PDE \eqref{PO} and heat equation. But, since simulation of heat transfer from the exterior to the interior of the component is necessary, we propose convective boundary conditions instead of the most common Dirichlet or Neumann conditions.

Let $ T:\overline{\Omega}\to R $ be a two times continuously differentiable heat distribution field solving 
\begin{align}
	\left. 
		\begin{array}{ll}
			 \Delta T=0 &\text{ in } \Omega\\[1ex]
			 k\dfrac{\partial T}{\partial \nu }=\eta(T_{e}-T) &\text{ on } \RO, \label{T}
		 \end{array} 
	\right.
\end{align}   
where $ T_{e}:\Oext\to \R{}  $, $ \Oext \supset \Omega$ denotes the component's ambient temperature, $ \eta:\RO \to R $ the heat transfer coefficient and $ k>0 $ the constant thermal conductivity.\\
According to \cite{HetEsl09}, taking into account a non constant temperature field $ T:\overline{\Omega} \to \R{} $ to \eqref{PO} leads to a new stress formulation
\begin{align}\label{Def: stress}
			\tilde{\sigma}(u)
			&= \lambda\Div(u)I+\mu(\nabla u+(\nabla u)^{\top})-\rho(3 \lambda+2\mu)(T-T_{0})I,
\end{align} 
where $ T_{0}\in \R{} $ is a reference temperature and $ \rho $ the coefficient of linear thermal expansion. If we insert the thermomechanical sress tensor field into  \eqref{PO} and rewrite it in term of it's mechanical component, the resulting combined equation reads

\begin{align}\label{PTO}
\begin{array}{c}
\begin{array}{ll}
\Delta T=0 &\text{ in } \Omega\\[1ex]
			 k\dfrac{\partial T}{\partial \nu }=\eta(T_{e}-T) &\text{ on } \RO,
\end{array}
\\\\
\begin{array}{ll}
\Div(\sigma(u))+f-\rho(3\lambda+2\mu)\nabla T=0 &\text{ in }\Omega\\
\se(u)=\lambda(\Div(u))I+\mu(\nabla u+(\nabla u)^{\top}) &\text{ in }\Omega\\
u = 0  &\text{ on } \partial\Omega_{D}\\
\sigma(u)\cdot \nu=g+\rho(3\lambda+2\mu)(T-T_{0})\cdot \nu &\text{ on }  \partial \Omega_{N}. 
\end{array}
\end{array}
\end{align}

\section{\label{sec:Basics}Basic Notations and Abstract Setting for Shape Design Problems}\label{Sec: basic notations}
In the following we are going to analyze generalized shape problems 
\begin{align}
\min ~& {\cal J}(\Omega) \nonumber \\
 s.t ~~&\Omega \text{ satisfies a given condition } P(\Omega), \label{P} \tag{$\mathbb{P}$}  \\
	&\Omega \in \O  \nonumber 
\end{align}
in order to adopt the universal concepts that are presentet e.g. in \cite{ShapeOpt} and apply the suggested solution strategy to the problem of LCF.

A solution of \ref{P} is sought as a set $\Omega$ in some family $\O$, called \textit{familiy of admissible domains}, containing possible candidates of shapes. First concentrate on the assumptions which have to be made on the \textit{cost functional}  $\mathcal{J}$, the restrictions $ P(\Omega) $ and the family $\O$. Therefore a summary of the approach presented in \cite{ShapeOpt} will be helpful. Further introductions to shape optimization can be found in \cite{DelfZol11, SokZol92, BucBut05}.

Let $\O$ be the set of admissible domains contained in a larger system $\Os$ on which we assume some kind of convergence, that has to be adjusted according to the respective problem. This convergence will be denoted by $\On \xrightarrow[]{\Os} \Omega $ as $n \to \infty$ for a sequence $\seq{\On}$ in $\Os$ and its limit $\Omega \in \Os$. Further we define a state space $V(\Omega)$ of real functions on $\Omega$ for every $\Omega \in \Os$ that contains possible solutions of $P(\Omega)$.\\
Since functions $y_{n} \in V(\On)$, $\On \in \Os$ are defined on changing sets, a suitable specification of convergence is necessary and has to be defined properly for each problem\footnote{See for example Definition \ref{Def: Conv. on var. domains}.}. Generally we write $y_{n} \rightsquigarrow y$ as $n \to \infty$. Moreover, require that any subsequence of a convergent sequence tends to the same limit as the original one. \\
In every $\Omega \in  \Os$ a \textit{state problem} $P(\Omega)$ has to be solved. This can be a PDE, ODE or variational inequality modeling for example forces that exert an influence on the component. Assuming that there is a unique solution $u(\Omega)$ for every state problem $P(\Omega)$ and every $\Omega \in \Os$ we are able to define the map $u:\Omega \to u(\Omega) \in V(\Omega)$. The resulting set $\G=\lbrace(\Omega,u(\Omega))\, \vert \, \Omega \in \O\rbrace$  is called the \textit{graph} of $u$ restricted to a chosen subfamily $\O$ of $\Os$. In this context we say that the graph $\mathcal{G}$ is compact iff every sequence $\seq{(\On,u(\On))} \subset \mathcal{G}$ has a subsequence $\subseq{\Omega_{n_{k}},u(\Omega_{n_{k}})}$ satisfying the condition 
\begin{align}
	\begin{array}{r c l}
	 \Omega_{n_{k}} &\xrightarrow[]{\Os} & \Omega \\
	  u(\Omega_{n_{k}}) &\rightsquigarrow & u(\Omega)
	\end{array}
	\label{Def: Gcomp}
\end{align}

\noindent as $ k \to \infty $ for some $(\Omega, u(\Omega)) \in  \mathcal{G}$.

\vspace{1em}

\noindent A \textit{cost functional} $\mathcal{J}$ on $\Os$ maps a pair $(\Omega,y)$,  $\Omega \in \Os$, $y \in V(\Omega)$ onto $ \mathcal{J}(\Omega,y)$, e.g. the functional introduced in Definition \ref{def:OptReliability}. Here, lower semi-continuity is defined as follows: 

\noindent Let the sequences $\seq{\On}$ in $\Os$ and $\seq{y_{n}},\, y_{n} \in V(\On)$ be convergent against $\Omega \in \O$ and $y \in V(\Omega)$, respectively. Then
\begin{align}
	\left.
	\begin{array}{r c l}
		\On & \xrightarrow[n \to \infty]{\Os} & \Omega \\
		 y_{n} &  \underset{ n \to \infty}{\leadsto}& y
	\end{array} \right\rbrace \Rightarrow \liminf_{n \to \infty} \mathcal{J}(\On,y_{n}) \geq \mathcal{J}(\O,y)
\end{align}

\noindent Now, let $\O$ be a subfamily of $\Os$ and let $u(\Omega)$ be the unique solution of a given state problem $P(\Omega)$ for every $\Omega \in \O$. An \textit{optimal shape design problem} can be defined by  
\begin{equation} \label{opt_shape_des}
	\text{Find } \Omega^{*} \in \O\text{ such that \ref{P} is solved}.
\end{equation}

\noindent The following theorem that can be found in \cite[Ch. 2]{ShapeOpt} provides conditions for the existence of optimal shapes. It is based on the general fact, that lower semicontinuous functions always possess a minimum on a compact set.

\begin{thm}\label{Thm: existence of optimal shapes}
Let $\Os$ be a family of admissible shapes with a subfamily $\O$. It is assumed that every $\Omega \in \O$ has an associated state problem $P(\Omega)$ with state space $V(\Omega)$ which is uniquely solved by $u(\Omega) \in  V(\Omega)$. Finally, require
\begin{itemize}
\item[(i)] compactness of $\mathcal{G}$,
\item[(ii)] lower semi-continuity of $\mathcal{J}$.
\end{itemize}  
Then there is at least one solution of the optimal shape design problem.
\end{thm}

\section{\label{sec:Domains}$ C^{k,\alpha} $-Admissible Domains via Defomation Maps} \label{Sec: form_design}

Now, we have to concretize the terms and results introduced in the last section and adjust them to our present problem. Among others, we have to define the family of admissible domains $\Os$ and the subfamily $\O$. Therefore, we consider $C^{k,\alpha}$-admissible domains\footnote{For explanations see Section 7 of \cite{Agm64} or Sections 4 and 6 of \cite{GilbTrud}} on which we later impose the boundary value problems of linear elasticity and heat equation that were introduced in Section \ref{sec:PDE}. We will see that $C^{k,\alpha}$-domains are very useful in relation to compactness properties. 

\begin{defn}[{\bf Basic-Design}] \label{Def: basic-design} Let $\Omega_{0} \subset \R{3}$ be a $C^{k,\alpha}$-domain for some  $\alpha \in (0,1]$. Further let $B:=\Umg{r}{z} \subset \Omega_{0}, \, z \in \text{int}(\hat{\Omega})$ be a ball in its interior having positive distance $D:=\dist(\Umg{r}{z}, \partial \hat{\Omega})>0$ from the boundary. Then we define the \textit{basic-design} by $\Omega_{b}:=\Omega_{0} \setminus B $.
\end{defn}

With regard to the following sections we state here that $\Omega_{b}$ in particular has a $C^{0,1}$-boundary and therefor satisfies a uniform cone property as explained in \cite{Chen75}. 
\pagebreak 
\begin{defn}[{ \bf $ C^{k,\alpha} $-Diffeomorphisms}]$  $ \vspace*{-0.2cm}
\begin{itemize}
\item[(i)]A $C^{k,\alpha}$-diffeomorphisms on a bounded domain $\Omega \subset \R{n},\, n\in \N{}$ is a one to one mapping $\Phi:\Omega \to \Omega',\, \Omega' \subset \R{n}$ such that $\Phi \in \Ck{k,\alpha}{\Omega}{n}$ and $\Phi^{-1}\in \Ck{k,\alpha}{\Omega'}{n}$. 
\item[(ii)]
The set of $C^{k,\alpha}$-diffeomorphisms on $\Omega \subset \R{n}$ will be denoted by $\left[\D{k,\alpha}{\Omega,\Omega'}\right]^n$ or $\left[\D{k,\alpha}{\Omega}\right]^n$ if $\Phi:\Omega \to \Omega$ \footnotemark.

\end{itemize}\end{defn} 

Let $\Omega_{b}$ be defined as declared in the upper definition. Applying only uniformly bounded Diffeomorphisms $ \Phi $, by means of $ \Norm{\Phi}{k,\alpha;\Oext}\leq K $ for some constant K, there is a Ball $ \Umg{R}{z}:=\Oext, \, R>r>0 $ containing all $ \overline{\Phi( \Omega_{b})} $.  

\begin{defn}[{\bf $C^{k,\alpha}$-Admissible Domains}] \label{Def: k-adm domains} 
In the present situation the elements of 
\[ 
\Uad_{k,\alpha}:=\left\lbrace \Phi \in \left[\D{k,\alpha}{\overline{\Oext}}\right]^3 \, \Big\vert \, \Norm{\Phi}{\Ck{k,\alpha}{\Oext}{3}}\leq K,\, \Norm{\Phi^{-1}}{\Ck{k,\alpha}{\Oext}{3}}\leq K \right\rbrace  
\] 
\noindent are called \textit{design-variables}. In a natural way, this induces the set of admissible shapes
\[ \O_{k,\alpha}:= \left\lbrace \Phi(\Omega_{b}) \, \vert \, \Phi \in \Uad_{k,\alpha} \right\rbrace \] 
assigned to $ \Omega_{b} $.
\footnotetext{ Note that $ \Phi \in  \Ck{k,\alpha}{\overline{\Omega}}{n} $ if $ \Phi \in   \Ck{k,\alpha}{\Omega}{n} $ has a $k,\alpha$-regular extension to $\overline{\Omega} $. }
\end{defn}
 
\begin{lem}\label{Lemma: compactnes Uad} 
$ \Uad_{k,\alpha} $ is compact in the Banach space $ \left(\Ck{k,\alpha'}{\overline{\Oext}}{3},\Norm{.}{\Ck{k,\alpha'}{\Oext}{3}} \right) $ for any  $ 0\leq \alpha' <\alpha $ and $ k \in \N{}$.
\end{lem}
 
\begin{proof} 
 The set $S:= \left\{ \Phi \in \Ck{k,\alpha}{\overline{\Oext}}{3} \, \Big\vert \, \Norm{\Phi}{\Ck{k,\alpha}{\Oext}{3}} \leq K \right\}$ is precompact in the Banach space $ \left(\Ck{k,\alpha'}{\overline{\Oext}}{3}, \Norm{.}{\Ck{k,\alpha'}{\Oext}{3}}\right) $ as stated in \cite[6.36]{RO}.
 \\[1ex] 
 Now let $ \Phi^{\ast} $ be the limit of a convergent sequence $ \seq{\Phi_{n}} $ in $ S $ regarding $ \Norm{.}{\Ck{k,\alpha'}{\Oext}{3}} $. We use the abbreviation $ \Norm{.}{k,\alpha'} $ for the $ \left[C^{k,\alpha'}\right]^3 $-Norms on $ \Oext $ and define  $ K_{n}:=\Norm{\Phi_{n}}{k} $, and $ \tilde{K_{n}}:= \sup_{x,y \in \Oext \atop \vert \beta \vert=k} \frac{\left|\abl{\beta}{\Phi_{n}}(x)-\abl{\beta} {\Phi_{n}}(y)\right|}{|x-y|^{\alpha}}$. From the inequality
 \[ 
 \Norm{\Phi_{n}}{k,\alpha}
 =\Norm{\Phi_{n}}{k}+\sup_{x,y \in \Oext \atop \vert \beta \vert=k} \frac{\left| \abl{\beta}{\Phi_{n}}(x)-\abl{\beta}{\Phi_{n}}(y)\right|}{|x-y|^{\alpha}}
 \leq K \]
 and the $ C^{k,\alpha'} $-convergence we conclude that $ \tilde{K}_{n}\leq K-K_{n} $ for all $ n\in \N{} $, as well as $ K_{n} \xrightarrow{n\to \infty} K^{\ast}=\Norm{\Phi^{\ast}}{k} \leq K$. 
 What remains to be shown is $ \left|\abl{\beta}{\Phi^{\ast}}(x)-\abl{\beta} {\Phi^{\ast}}(y)\right|  \leq (K-K^{\ast})|x-y|^{\alpha}  $ holds for every $ |\beta |=k $, $ x,y\in \Oext $. We apply triangle inequality to the left hand term and receive
 \begin{equation*}
  \begin{split}
   \left|\abl{\beta}{\Phi^{\ast}}(x)-\abl{\beta} {\Phi^{\ast}}(y)\right| 
   \leq    \left|\abl{\beta}{\Phi^{\ast}}(x)-\abl{\beta} {\Phi_{n}}(x)\right|
   		   + &\left|\abl{\beta}{\Phi_{n}}(x)-\abl{\beta} {\Phi_{n}}(y)\right| \\
   		   &~~~~~~~~~+ \left|\abl{\beta}{\Phi_{n}}(y)-\abl{\beta} {\Phi^{\ast}}(y)\right|.  
  \end{split}
\end{equation*}
Because of the uniform convergence of the $ k $-th order derivatives, the first and the third term on the right hand side become zero as $ n\to \infty $. The second one can be estimated by $ \left|\abl{\beta}{\Phi_{n}}(x)-\abl{\beta} {\Phi_{n}}(y)\right| \leq (K-K_{n})|x-y|^{\alpha} $ what leads to $ \left|\abl{\beta}{\Phi^{\ast}}(x)-\abl{\beta} {\Phi^{\ast}}(y)\right| \leq (K-K^{\ast})|x-y|^{\alpha} $ when passing to the limit. Therefore $ S $ closed. The letter is also true for $ S^{-1}:=\left\{ \Phi^{-1} \, \vert \, \Phi \in S \right\} $. Since $ \Phi_{n}\circ \Phi_{n}^{-1}=\Phi_{n}^{-1}\circ \Phi_{n}=id $ for all $ n\in \N{} $ and $ \Phi_{n} \in \Uad_{k,\alpha} $ it holds that $ \lim\limits_{n \to \infty} \Phi_{n}^{-1}=\Phi^{\ast} $ if $ \Phi^{*}=\lim\limits_{n \to \infty} \Phi_{n} $. Thus $ \Uad_{k,\alpha}$ is a closed subset of the compact set $ S $ and the statement holds. \hfill $\square$
\end{proof}

\noindent In this context it is obvious to define convergence of sets through $ \left[C^{k,\alpha'}\right]^3 $-convergence of admissible functions.

\begin{defn}[{\bf $ C^{k,\alpha'} $-Convergence of Sets}] 
Let $ 0\leq \alpha' \leq \alpha $ be fixed. Then $ \Omega_{n} \overset{\O}{\longrightarrow} \Omega, \, n \to \infty $ iff there is a sequence $ \seq{\Phi_{n}} \subset \Uad_{k,\alpha} $, $ \Phi \in \Uad_{k,\alpha} $ where $ \Phi_{n}(\Omega_{b})=\Omega_{n} \, \forall n\in \N{}$, $ \Phi(\Omega_{b})=\Omega $ and  $ \Phi_{n} \to \Phi, \, n \to \infty $ in $ \Ck{k,\alpha'}{\overline{\Oext}}{3} $. The set $ \Omega_{b} $ shall be defined as in Definition \ref{Def: basic-design}.
\end{defn}

\begin{rem}[Volume Constraints]\label{rem:VolumeConstraints}
One can easily restrict the set of admissible domains with geometric constraints. Let us take the volume constraint as an example and let $V=\int_{\Omega_b}\, dx$ be the volume of the baseline design  $\Omega_b$. We then consider only those deformation maps $\Phi\in U_{k,\alpha}^{ad}$ that preserve the volume of $\Omega_b$, i.e.
\begin{equation}\label{eq: vol Constraint}
V=\int_{\Omega_b}\, dx=\int_{\Phi(\Omega_b)} \, dx=\int_{\Omega_b} \left|\det(\nabla \Phi)\right|\,dx
\end{equation}
Let $U_{k,\alpha,V}^{ad}$ be the set of all $\Phi\in U_{k,\alpha}^{ad}$ that fulfil \eqref{eq: vol Constraint} and let . Form this equation it is clear that $U_{k,\alpha,V}^{ad}$ is closed in $U_{k,\alpha}^{ad}$ (if $k\geq 1$) and therefore compact in the $C^{k,\alpha'}$-topology for $\alpha'<\alpha$.

Taking this into account, we see that all arguments of this article are equally valid for the set of admissible shapes ${\cal O}_{k,\alpha,V}=\{\Phi(\Omega_b):\Phi\in U_{k,\alpha,V}^{ad}\}$ with volume constraint $V$.

\end{rem}

\section{\label{sec:UnifSchauder}Uniform Schauder Estimates}\label{Sec: uniform schauder estimates}

Recall the mixed problem \eqref{PTO} in Section \ref{Sec: PDE constraints} and the definitions of design variables and admissible shapes in Section \ref{Sec: form_design}. 

At the end of this paper we want to apply Theorem \ref{Thm: existence of optimal shapes}: In this Section, suitable assumptions on $ \O $ and on the appearing functions in \eqref{PO} and \eqref{T} will be presented, such that the requirement of unique solubility of \eqref{PTO} is satisfied. This ensures the existence of the graph $ \mathcal{G}=\{ (\Omega,T(\Omega),u_{T}(\Omega)) \, \vert \, \Omega \in \O \} $, as claimed. We will also see, that under appropriate assumptions the resulting solutions $ u_{T} $ and $ T $ are H\"older functions, what ensures a proper definition of convergence of solution sequences $ \seq{u(\Omega_{n})} $ and $ \seq{T(\Omega_{n})} $.
The crucial step, when showing compactness of $ \mathcal{G} $ (see Lemma \ref{Lemma: compactness of the graph}) in the sense of  \eqref{Def: Gcomp}, is the application of Schauder estimates to proof the solution's uniform boundedness with respect to $ \O $. This gives us the possibility to apply Lemma \ref{Lemma: compacness of c_k_alpha} and leads to the desired conclusion. Finally we will show lower-semicontinuity in Lemma \ref{Lemma: compactness of the graph} for a very general class of cost functionals containing ours.

\subsection{Schauder estimates for Linear Elasticity Equation}\label{Subsec: solutions for elasticity equation}

We start with an review of regularity results for the disjoint displacement-traction problem of linear elasticity presented in \cite{Hanno} and point out the main characteristics of the shape's geometry that lead to the uniformity of certain estimates. Since $C^{4}$-regularity is needed in Theorems 6.3-5 and 6.3-6 in \cite{Cia1988} which are used in the proof of Theorem 5.1 et seq. in \cite{Hanno}. Accordingly, we set $ \Uad:=\Uad_{4,\alpha} $, $ \alpha \in (0,1) $ for the set of feasible design-variables and  $ \O:=\O_{4,\alpha} $ for the set of  admissible shapes.


\begin{lem} \label{Lemma: uniform hemisphere condition}
Each $ \Omega \in \O $ satisfies a hemisphere property where the corresponding hemisphere transformations are of class $ C^{4,\alpha'} $, $ \alpha' \in [0,\alpha] $ and have a uniform bound $ \mathcal{K} $ with respect to $ \O $.
\end{lem}


\begin{proof}

First we note that $ \Omega_{b} $ satisfies a hemisphere condition, see Definition \ref{Def: hemisphere transformation}. This can be proved analogously to Lemma 5.4 in \cite{Hanno} because $ \Omega_{b} $ is a $ C^{4,\alpha} $ domain and thus it's compact boundary can be parametrized by a finite family of uniformly bounded $ C^{4,\alpha} $-functions in two variables. The resulting hemisphere transformations depend on the point $ z_{0}\in \Omega_{b} $ lying in a sufficient small distance $ 0<d $ of the boundary, that depends on the curvature of the boundary.


Since every $ \Phi \in \Uad $ is a $ C^{4,\alpha} $-diffeomorphism, the compositions $ \T_{\Phi(z_{0})}:=\Phi \circ \mathcal{T}_{z_{0}} $ are again hemisphere transformations: The functions $ \Phi \in \Uad $ are one to one mappings from $ \Omega_{b} $ to $ \Phi(\Omega_{b}) $ that are uniformly bounded. Therefore we have a constant $\mathcal{ K}>0 $, depending on $ \Omega_{b} $ but not on the choice of $ \Phi $, where 
 
 \begin{equation}\label{eq: (6.29) Gilbarg/Trudinger}
 \mathcal{ K}^{-1}\Norm{x-y}{\R{3}}  \leq \Norm{\Phi(x)-\Phi(y)}{\R{3}}\leq \mathcal{ K} \Norm{x-y}{\R{3}}
 \end{equation}

\noindent for all $ x,y \in \Omega_{b} $, see (6.29) in \cite{GilbTrud}. As a consequence we set $ d'=K^{-1}d $ (uniformly for all $\Omega \in \O$) and construct the new neighborhood $ U' $ as a proper extension of $ \Phi(U\cap\Omega_{b}) $ beyond the boundary.
At last we apply chain rule and see that the transformations $\T_{\Phi(z_{0})} $ are even uniformly bounded in the $ C^{4,\alpha'} $-norms, $ 0\leq \alpha' \leq \alpha $ with respect to $ \O $ because the functions $ \Phi $ and $ \T_{z_0} $ are. This is also true for the inverse functions by analogous arguments.
\end{proof}

\begin{exm}
Let $ O \subset \R{3} $ be a $ C^{1} $-domain where $ \psi: \Umg{a}{0}\subset \R{2}\to \R{}$ describes a part of the upper boundary. Then $ \T:\Sigma_{a}\subset \R{3} \to \tilde{\Sigma} \subset\Omega  $ where
\begin{equation}\label{eq: hemisphere Transformations Omega_b }
\T(z_{1},z_{2},z_{3}):=
\begin{pmatrix}
z_{1} \\
z_{2}\\
\psi (z_{1},z_{2})-z_{3}
\end{pmatrix},
\end{equation}
defines a one to one mapping from the half ball $ \Sigma_{a} $ with radius $ a $ at $ 0 \in \R{3} $ to $\tilde{\Sigma}_{a}$
\end{exm}

Due to it's construction, every shape $ \Omega \in \O $ has a Lipschitz-boundary and the associated Lipschitz constant can be chosen uniformly, what is proofed to be equivalent to a uniform cone property in \cite{Chen75}. Hence, the following Lemma is applicable:

\begin{lem} \cite[Lemma 5.5]{Hanno} \label{Lemma: uniform cone lemma} 
Let $ \mathcal{M} $ be a set of bounded domains in $ \R{n}$ with a uniform cone property and let $ \Omega \in \mathcal{M} $. Then, for every $ \varepsilon>0 $ there is a constant $ C(\varepsilon)>0 $ uniform with respect to $ \mathcal{M} $, such that  $ \Norm{v}{\C{0}{\Omega}} \leq \varepsilon \Norm{v}{\C{1}{\Omega}} + C(\varepsilon)\int_{\Omega}\vert v \vert dx  $ holds for all $ v \in \C{1}{\Omega} $.
\end{lem}

One of the complications of setting up a realistic shape optimization problem is the definition of the surface force density $g$. While the volume force densities are  easily defined as gravitational or centrifugal loads, the surface load $g$ generally depends on the shape $\Omega$ in a non trivial way. Often, $g=g(\Omega)$ will be defined by the static pressure of a fluid surrounding $\Omega$. In this paper we do not intend to give a solution to this complicated problem. The following definition sets a framework that is capable to deal with such effects.

\begin{defn}[Admissible Surface Force Model]
\label{Def: SurfaceForceModel}
    Let $ G(\mathcal{O},\phi):=\dot\bigcup_{\Omega\in\mathcal{O}} \Ck{2,\phi}{\RO}{3}$ be the vector bundle with fiber $\Ck{2,\phi}{\RO}{3}$ over $\Omega\in\mathcal{O}$. 
    We define the space of admissible surface force models as a space of sections in with uniform bound on the fibre norm, $ G^{ad}(\mathcal{O},\phi):=\lbrace \bar{g}:\O \to G(\Omega,\phi)\mbox{ s.t. }\bar g(\Omega)\in \Ck{2,\phi}{\RO}{3}\mbox{ and } \exists 0<k_1<\infty \mbox{ s.t. } \Norm{\bar g(\Omega)}{\Ck{2,\phi}{\RO}{3}}\leq k_{1} ~\forall\Omega\in{\cal O}\rbrace $.   
\end{defn}
With every $ \bar{g} \in G^{ad} $ we can thus associate surface force boundary conditions $ \bar{g}(\Omega)$ to any set $\Omega\in\mathcal{O}$ that have a uniform common bound on their $\Ck{2,\phi}{\RO}{3}$ norm. The following example has been used in \cite{Hanno}:

\begin{exm}
Let $ g^{ext}\in \Ck{2,\phi}{\Oext}{3} $ be an arbitrary mapping. Then we can define $ \bar{g}\in G^{\rm ad}(\mathcal{O},\phi)$ by $ \bar{g}(\Omega):=g^{ext}\restriction_{\RO} $ with $k_1=\Norm{g^{ext}}{\Ck{2,\phi}{\RO}{3}}$.
\end{exm}

\begin{thm} \label{Thm: schauder estimates u} \cite[Theorem 5.6, 5.7]{Hanno} 
Recall the PDE \eqref{P}, where $ \Omega=\Phi(\Omega_{b}) $ for some $ \Phi \in \Uad $.
\begin{itemize}
\item[(i)] Let $ f\in \Ck{1,\phi}{\overline{\Oext}}{3} $, $ g\in \Ck{2,\phi}{\overline{\Omega}}{3}$ \footnote{$\C{k,\alpha}{\overline{\Omega}}$ and $\C{k,\alpha}{\Omega}$-functions can be identified and therefore replaced by each other, see \cite{GilbTrud} }  for some $ \phi\in (0,1) $. Then there exists unique solution $ u\in \Ck{3,\phi}{\overline{\Omega}}{3} $ that satisfies 
\begin{equation}\label{eq: schauder inequality u}
\Norm{u}{\Ck{3,\varphi}{\Omega}{3}} \leq C \bigl(  \Norm{f}{\Ck{1,\phi}{\Omega}{3}} + \Norm{g}{\Ck{2,\phi}{\RO}{3}} +  \Norm{u}{\Ck{0}{\Omega}{3}}  \bigr). 
\end{equation}

\noindent for any $ \varphi \in (0,\phi) $ and some positive constant $ C $ independent from $ \Omega \in \O $.

\item[(ii)] Let $ f\in \Ck{2,\phi}{\overline{\Oext}}{3} $. Moreover, let $ g=\bar{g}(\Omega) $ be the associated mapping to some $ \bar{g} \in G^{ad} $. Then, by means of Lemma \ref{Lemma: uniform cone lemma} the term $ \Norm{u}{\Ck{0}{\Omega}{3}} $ can be replaced by $ \int_{\Omega} \vert u \vert \, dx $ and even  
\begin{equation}
\Norm{u}{\Ck{3,\varphi}{\Omega}{3}} \leq C^{el}
\end{equation} 
holds for any $ \varphi \in (0,\phi) $ and a constant $ C^{el} $ which can be chosen uniformly w.r.t. $ \O $.
\end{itemize}
\end{thm} 

\begin{proof}
The proof is essentially the same as in  Theorem 5.6 in \cite{Hanno}. We only have to show that it also holds for the extended set of geometries $\mathcal{O}=\mathcal{O}_{k,\alpha}$ that is considered in this paper. Two aspects of the geometry definition are relevant for the uniform Schauder estimates: Uniform bounds $\kappa$ of hemisphere transformations that (locally) straighten out the boundary $\partial\Omega$ are build in our definition of admissible shapes $\mathcal{O}$.  In fact, let $\Phi_i :U_i\to \Sigma$ be a finite collection of $C^{4,\alpha}$-hemisphere transformations on $\Omega_b$ such that $U_i$ cover $\partial\Omega$. Obviously, the $C^{4,\alpha}$ norms of $\Phi_i$ and $\Phi_i^{-1}$ are uniformly bounded. However, for $\Omega=\Phi(\Omega)$ with $\Phi\in U^{\rm ad}_{k,\alpha}$, the same applies to the sets $U_i'=\Phi(U_i)$ with hemisphere transformations $\Phi_i'=\Phi_i\circ\Phi$. Clearly, the  $C^{4,\alpha}$-norms of these transformations an their inverses only depend on the related norms of the $\Phi_i$ and the constant  $K$ used in the definition of $U_{k,\alpha}^{ad}$.

We also need a uniform lower bound $ \delta^{\ast}>0$ of the radii of the balls on which the boundary is straightened by hemisphere transformations. This bound $ \delta^{\ast}$ can be constructed as shown in  Lemma  \ref{Lemma: Schauder estimate T} below.
\end{proof}

\subsection{Schauder Estimates for the Heat Equation} \label{Subsec: solutions for heat equation}
Recall the heat equation \eqref{T} presented in Section \ref{Sec: PDE constraints}. By application of Theorem 6.31 in \cite{GilbTrud} it is easy to see, that for any $ \phi\in(0,1) $ there is a unique solution $ T\in C^{2,\phi}(\overline{\Omega}) $ of \eqref{T} for every $ \Omega \in \O $ supposed that $ \eta \in \C{1,\phi}{\overline{\Omega}}$, $ \eta\restriction_{\RO}>0 $ and $ T_{e}\in \C{1,\phi}{\overline{\Omega}} $.


Assume that $ T \in \C{2,\phi}{\overline{\Omega}} $ is a solution of \eqref{T} for some $ \Omega \in \O $. 
 Then we can apply inequality \eqref{eq: schauder inequality u} to \eqref{PTO} which is equivalent to the problem \eqref{PO} with load vector fields $ \tilde{f}=f-\rho(3\lambda+2\mu)\nabla T $ and $ \tilde{g}=g+\rho(3\lambda+2\mu)(T-T_{0})\cdot \nu $. Together with triangle inequality and Lemma \ref{Lemma: uniform cone lemma} we obtain for the unique solutions $ u $ of \eqref{PO} and $ u_{T} $ of \eqref{PTO}
\begin{equation*} 
\begin{split} 
\Norm{u_{T}}{\Ck{3,\varphi}{\Omega}{3}}
&\leq C\left( \Norm{f-c\nabla T}{\Ck{1,\phi}{\Omega}{3}}+ \Norm{g-c(T-T_{0})\cdot \nu}{\Ck{2,\phi}{\RO}{3}}\vspace{0.1cm}+\Norm{u}{\Ck{0}{\Omega}{3}}\right)
\\
 & \leq C\bigl( \Norm{f}{\Ck{1,\phi}{\Omega}{3}}+\Norm{g}{\Ck{2,\phi}{\RO}{3}}+\varepsilon\Norm{u}{\Ck{1}{\Omega}{3}}
+C(\varepsilon)\int_{\Omega} \vert u \vert \, dx\bigr)
\\
& ~~~+C\rho(3\lambda+2\mu)\left(\Norm{\nabla T}{\Ck{1,\phi}{\Omega}{3}} +\Norm{T-T_{0}}{\C{2,\phi}{\RO}}\Norm{\nu}{\Ck{2,\phi}{\RO}{3}}\right)
\end{split}
\end{equation*}
\noindent where $ c:=\rho(3\lambda+2\mu) $. Thus, 
\begin{align}\label{eq: estimate u_{T}}
(1-\varepsilon C) \Norm{u_{T}}{\Ck{3,\varphi}{\Omega}{3}} \leq C_{f,g}+ C\rho(3\lambda+2\mu)&\left(\Norm{\nabla T}{\Ck{1,\phi}{\Omega}{3}}+\right. \nonumber\\ 
&~~\left.\Norm{T-T_{0}}{\C{2,\phi}{\RO}}\Norm{\nu}{\Ck{2,\phi}{\RO}{3}}\right).
\end{align} 
As explained in the proof of Theorem 5.7 in \cite{Hanno} it can be shown by two times application of Korn's second inequality to the $ L^{1} $-Norm of $ u $ that the constant $ C_{f,g}$ can be chosen to be uniform w.r.t $ \O $. The letter is also true for  $\Norm{\nu}{\Ck{2,\phi}{\RO}{3}} \leq \tilde{C}$, confer the proof of Lemma \ref{Lemma: Schauder estimate T}.
Hence, it will be sufficient to show that $\Norm{T}{\C{2,\phi}{\Omega}} \leq C^{t}$ for a constant $C^{t}$ independent of $ \Omega \in \O $ to derive uniform boundedness of the solutions $ u $ of \eqref{PTO}.

The following result is presented by not proofed in detail by \cite{GilbTrud}. In order to proof the uniformity of the constant $ C^{t} $ occuring in the subsequent Lemma, we transfer the strategy they present in \cite{GilbTrud} Lemma 6.5 and Theorem 6.6:

\begin{thm}[\textbf{\textit{Schauder Estimates for Elliptic PDE with Convective BC}}]\label{Thm: 6.30 Gilbarg/Trudinger}$  $\\
Let $ \Omega $ be a $ C^{2,\phi} $ domain in $ \R{n} $, and let $ T\in \C{2,\phi}{\overline{\Omega}} $ be a solution of $ L(x)T=f $. Define the boundary condition (BC) by
$$
B(x)T \equiv \gamma(x)T +\sum_{i=1}^{n} \beta_{i}(x) D^{i}T = \vartheta(x) ,\,x\in \RO,
$$ 
where the normal component $ \beta_{\nu} $ of the vector $ \beta=(\beta_{1},...,\beta_{n}) $ is non zero and $ \Norm{\beta_{\nu}}{\R{n}}\geq \kappa >0 \text{ on } \RO,$ for some constant $ \kappa $. It is assumed that the operator $ L $, defined by
$$
L(x)T=a_{i,j}(x)D^{i,j}T+b_{i}(x)D^{i}T+c(x)T,\, a_{i,j}=a_{j,i},\, i,j=1,\ldots,n,
$$
is strictly elliptic with constant $ l $ and that 
$f \in \C{\phi}{\overline{\Omega}},\, \vartheta \in \C{1,\phi}{\overline{\Omega}}, \, a_{i,j},\,b_i,\, c \in \C{\phi}{\overline{\Omega}} $
and $ \gamma,\, \beta_{i} \in \C{1,\phi}{\overline{\Omega}} $ with 
$$
\Norm{a_{i,j},\,b_{i},\,c}{\C{0,\phi}{\Omega}}, \, \Norm{\gamma,\, \beta_{i}}{\C{1,\phi}{\Omega}}\leq \mathscr{L},\,i,j=1,\ldots,n.
$$ 
Then 
\begin{equation}
 \Norm{T}{\C{2,\phi}{\Omega}}\leq C \left(\Norm{T}{\C{0}{\Omega}}+\Norm{\vartheta}{\C{1,\phi}{\Omega}}+\Norm{f}{\C{0,\phi}{\Omega}}\right)
\end{equation}
where $ C=C_{n,\phi,l,\mathscr{L},\kappa,\Omega} $.
\end{thm} 

\begin{proof}
We choose a $ C^{2,\phi} $-diffeomorphism $ \tau $ that straightens the boundary in a neighborhood $ N $ of a point $ x_{0}\in \RO $. Let $ \Umg{\delta}{x_{0}} \subset \subset N$ and set 

\begin{equation}\label{Choice B_{0}}
 B_{0}=\Umg{\delta}{x_{0}}\cap\Omega,~~~~ \Gamma_{0}=\Umg{\delta}{x_{0}}\cap\RO \subset \partial B_{0} . 
\end{equation}
 
 \noindent Consider the local problem
\begin{align*}
	\left. 
		\begin{array}{ll}
			 ~~~~~L T=0 &\text{ in } B_{0}\\[1ex]
			 \gamma(x) T_{e} =\gamma(x) T+\sum_{i=1}^{n}\beta_{i}(x)D^{i} T &\text{ on } \Gamma_{0},
		 \end{array} 
	\right.
\end{align*}

\noindent which is transformed to  
\begin{align*}
	\left. 
		\begin{array}{ll}
			 ~~~~~\tilde{L} \tilde{T}=0 &\text{ in } \tilde{B}_{0}\\[1ex]
			 \tilde{\gamma}(y) \tilde{T}_{e} =\tilde{\gamma}(y) \tilde{T} +\sum_{i=1}^{n}\tilde{\beta}_{i}(x)D^{i} \tilde{T} &\text{ on } \tilde{\Gamma}_{0},
		 \end{array} 
	\right.
\end{align*}

\noindent by $ \tau $. For $ y=\tau(x),\,\tilde{T}(y)=T(x) $ the equation $ \tilde{L}(y) \tilde{T}=\tilde{a}_{i,j}(y)D^{i,j}\tilde{T}+\tilde{b}_{i}(y)D^{i}\tilde{T}+\tilde{c}(y)\tilde{T} $ \footnote{For a detailed description see the proof of Lemma 6.5 \cite{GilbTrud}} holds for $ y \in \tilde{B}_{0}=\tau(B_{0}) $. This defines again an elliptic PDE: Since $ \tilde{a}_{i,j}(y)=\sum_{i,j=1}^{n}D^{r}\tau_{i}(x)D^{s}\tau_{j}(x)a_{r,s}(x) $ we receive for $ x\in B_{0},\,\xi\in \R{n} $
\[
\sum_{i,j=1}^{n}\tilde{a}_{i,j}\xi_{i}\xi_{j}
=\sum_{r,s=1}^{n}a_{r,s}(D^{r}\tau_{i}(x)\xi_{i})(D^{s}\tau_{j}(x)\xi_{j})
\geq l \Norm{J_{\tau}(x)\xi}{\R{3}}^2
=l \xi^{\top} J_{\tau}(x)^{\top}J_{\tau}(x)\xi.
\]

\noindent Since $ \tau $ is an $ C^{2,\phi} $-diffeomorphism, it's determinant is nowhere equal to zero and $ J_{\tau}(x)^{\top}J_{\tau}(x) $ is symmetric, positive definite for any $ x\in B_{0} $. Hence,  $ J_{\tau}^{\top}J_{\tau}=(S\mathcal{D}^{1/2})^{\top}\mathcal{D}^{1/2}S $ on $ B_{0} $, where $ S $ is a suitable orthogonal matrix and $ \mathcal{D} $ is a diagonal matrix whose diagonal elements are the positive eigenvalues $ \lambda^{\tau}_{i}, \, i=1,\ldots,n $ of $ J_{\tau}^{\top}J_{\tau} $. We conclude 
\[ 
\sum_{i,j=1}^{3}\tilde{a}_{i,j}(y)\xi_{i}\xi_{j} 
\geq l\Norm{\mathcal{D}^{1/2}(x)S(x)\xi}{\R{3}}^{2}
\geq l\, \lambda^{\tau}_{min}(x)\Norm{\xi}{\R{3}}^{2}~~~ x\in B_{0},\, \xi\in \R{n}.
\]

\noindent The eigenvalues depend continuously on $ x $ lying in the compact set $\overline{B_{0}}$, what gives us the possibility to choose a lower bound $l\, \lambda^{\tau}_{min}(x)\geq\tilde{l}>0 $.\\
By application of (6.30) \cite{GilbTrud} one sees that $ \Norm{\tilde{a}_{i,j}}{\C{0,\phi}{B_{0}}}, \, \Norm{\tilde{b}_{i}}{\C{0,\phi}{B_{0}}}, \Norm{\tilde{c}}{\C{0,\phi}{B_{0}}}\leq \tilde{\mathscr{L}}=C_{\tau}\mathscr{L} $, because $\tau $ is bounded upwards in $ \Norm{.}{\C{2,\phi}{\Omega}} $. The letter is also true for 
$ \Norm{\tilde{\gamma}}{\C{1,\phi}{\tilde{B}_{0}}} $, where $ \tilde{\gamma}(y)=\gamma(x) $, and $ \Norm{\tilde{\beta_{i}}}{\C{1,\phi}{\tilde{B}_{0}}} $ with $ \tilde{ \beta_{i}}(y) = \sum_{j=1}^{n} \beta_{j}(x)D^{j} \tau_{i}(x) $ on the hyperplane portion $ \tilde{\Gamma}_{0}=\tau(\Gamma_{0}) $.
Now we will show that $ \vert \tilde{\beta}_{n} \vert \geq \tilde{\kappa}>0 $, what is requested for the statement in Lemma 6.29\footnote{Lemma 6.29 is the analogous one to Lemma 6.4, which is used in the proof of Lemma 6.5 } \cite{GilbTrud}. \\
Since $ \tau:\Omega\subset \R{n} \to \Sigma_{R} \subset \R{n}$ we can consider the differential $ d \tau_{x} $ as a mapping  from $ T_{x}\subset \R{n} $ to $ \R{n} $ and interpret $ \RO $ in the $ n $ dimensional sense. We choose an orthonormal basis (ONB) of vectorfields $(e_{1}'(x),\ldots,e_{n-1}'(x))$ of the tangential space assigned to the $ n-1 $ dimensional submanifold $\RO $ at $ x\in \Gamma_{0} $ and extended it to an ONB $ E'(x) $ of $ \R{n} $ by $ e_{n}'(x)=\nu(x) $. Furthermore, define $ E=(e_{1},\ldots,e_{n}) $ to be the standard ONB of $ \R{n} $. Then, it holds that

\begin{equation*}
\begin{split}
\vert \tilde{\beta}_{n}(y) \vert 
=& \vert e_{n} d \tau_{x} \beta(x) \vert
= \vert e_{n} d \tau_{x} \left( \Norm{\beta_{\nu}(x)}{\R{n}} \nu(x)+\beta_{T} \right)  \vert \\
=&\Norm{\beta_{\nu}(x)}{\R{n}} \left|e_{n} d \tau_{x} \nu(x)\right|
\geq \kappa \left|e_{n} d \tau_{x} e_{n}'(x)\right|.
\end{split}
\end{equation*}

\noindent We contemplate the matrix $ B_{i,j}=e_{i} d\tau e_{j}',\, i,j=1,\ldots,n $ for $ x \in \Gamma_{0} $. Because both, $E $ and $ E'(x) $, are ONB of the $ \R{n} $ there exists an orthogonal matrix $ O(x),\, x\in \Gamma_{0} $ where $ e_{i}'(x)=O(x)e_{i},\, \forall \, x\in \{1,\ldots,n\} $. Thereby, we follow
\[ 
B_{i,j}
=\sum_{k=1}^{n}e_{i} d\tau e_{k}O_{k,j}
=(J_{\tau}O^{\top})_{i,j},\, i,j=1,\ldots,n
\]

\noindent and 
\[ 
\left|\det(B(x))\right|
=\left| \det\left(J_{\tau}(x)O(x)^{\top}\right) \right|
=\left|\det\left(J_{\tau}(x)\right)\right|
\geq C_{1} >0
\]

\noindent On the other hand, $ B_{n,j}=e_{n} d\tau e_{j}'=e_{n}e_{j}=0 $ for every $ j\neq n $ and $ B_{n,n}= e_{n} d\tau e_{n}'$. Thus, the matrix $ B $ has the structure
\begin{equation*}
B=\left(	 	
  \begin{array}{c|c}
	 \left(J_{\tau}(x)O(x)^{\top}\right)_{i,j=1,\ldots,n-1} & \vec{b}  \\
	 \hline \vspace{-0.5cm}\\
	 \vec{0}^{\,\top} & e_{n} d\tau e_{n}' \\
  \end{array}
  \right),~~~
  \vec{b},\vec{0} \in \R{n-1},
\end{equation*}

\noindent and it's determinant can be calculated by
\[ 
\left|\det(B(x))\right| 
=\left| \det\left(J_{\tau}(x)O(x)^{\top}\right)_{i,j=1,\ldots,n-1} \right| \left|e_{n}d\tau_{x}e_{n}'(x)\right|
\leq C_{2} \left|e_{n}d\tau_{x}e_{n}'(x)\right|.
\]

\noindent In consequence $ C_{2} \left|e_{n}d\tau_{x}e_{n}'(x)\right| \geq C_{1} >0 ~ \Leftrightarrow ~\left|e_{n}d\tau_{x}e_{n}'(x)\right| \geq \frac{C_{1}}{C_{2}}:=C >0  $ and finally 
\[
\vert \tilde{\beta}_{n}(y) \vert  \geq \kappa C:=\tilde{\kappa}>0.
\]

\noindent All conditions requested in the mentioned Lemma 6.29 are accomplished, what yields 
\[ 
\Norm{\tilde{T}}{\C{2,\phi}{\tilde{B}_{0}\cup \tilde{\Gamma}_{0}}}^{\ast}
\leq \tilde{C} \left( \Norm{\tilde{T}}{\C{0}{\tilde{B}_{0}}}
+\Norm{\tilde{\vartheta}}{\C{1,\phi}{\tilde{\Gamma}_{0}}}
+ \Norm{\tilde{f}}{\C{0,\phi}{\tilde{B}_{0}}}\right),
\]

\noindent with a constant $ \tilde{C} $ that depends on $n,\,\phi,\,\tilde{l},\,\tilde{\mathscr{L}},\,\tilde{\kappa},\,\diam(\tilde{B_{0}}) $. Now we use exactly the same arguments which are carried out in \cite{GilbTrud} and obtain for $ B_{0}'=\Umg{\delta/2}{x_{0}} $ 
\begin{equation}\label{eq: 1. local inequality T}
\Norm{T}{\C{2,\phi}{B_{0}'}}\leq C \left( \Norm{T}{\C{0}{B_{0}}}
+\Norm{\vartheta}{\C{1,\phi}{\tilde{\Gamma}_{0}}}
+ \Norm{f}{\C{0,\phi}{B_{0}}} \right)
\end{equation}

\noindent where $ C=C_{n,\phi,l,\mathscr{L},\kappa,\diam(B_{0}),\tau} $ due to the construction of the coefficients $ \tilde{l},\,\tilde{\mathscr{L}},\,\tilde{\kappa} $, the structure of set $ \tilde{B}_{0} $ and the statement (6.30).\\
Since the boundary of $ \Omega $ is compact one needs only a fixed number $ m $ of points $ x_{i} $ and radii $ \delta_{i} $ to cover the whole boundary. We choose $ \delta^{\ast}=\min \delta_{i}/4 $, $ B=\Umg{ \delta^{\ast}}{x_{0}} $ and assert that 

\begin{equation}\label{eq: 2. local inequality T}
\Norm{T}{\C{2,\phi}{B\cap\Omega}}\leq C\left(\Norm{T}{\C{0}{\Omega}}
+\Norm{\vartheta}{\C{1,\phi}{\RO}}
+ \Norm{f}{\C{0,\phi}{\Omega}}\right)
\end{equation}

\noindent for the maximum  $C= C_{i} $ that is assigned to $ x_{i},\, i=1,\ldots,m $ appearing in \eqref{eq: 1. local inequality T}. Therefore, $ C=C_{n,\phi,l, \mathscr{L},\kappa, \Omega} $, where the dependence on $ \Omega $ is through the radius $  \delta^{\ast} $, the transformations $ \tau $ and $ \diam(\Omega) $.\\ The remainder of the proof is essentially the same as in Theorem 6.6 in \cite{GilbTrud}, but with \eqref{eq: 2. local inequality T} instead of (6.32). Using the same distinction of cases, we end up with 
\begin{equation}
\Norm{T}{\C{2,\phi}{\Omega}}\leq C \left(\Norm{T}{\C{0}{\Omega}}+\Norm{\vartheta}{\C{1,\phi}{\Omega}}+\Norm{f}{\C{0,\phi}{\Omega}}\right)
\end{equation}

\noindent where $ C=C_{n,\phi,l, \mathscr{L},\kappa, \Omega} $ and the dependence of $ \Omega $ is subject to $  \delta^{\ast} $,  $ \tau $  and $ \diam(\Omega) $.
\end{proof}

\begin{defn}\label{Def: heat transfer coeff.}
	Analogously to Definition \ref{Def: SurfaceForceModel}, let $ \mathcal{E}(\mathcal{O},\phi):=\dot\bigcup_{\Omega\in\mathcal{O}} \C{1,\phi}{\RO}$ be the vector bundle with fiber $\C{1,\phi}{\RO}$ over $\Omega\in\mathcal{O}$. 
	We define the space of admissible heat transfer functions as $ \mathcal{E}^{ad}(\mathcal{O},\phi):=\lbrace \bar{\eta}:\O \to \mathcal{E}(\Omega,\phi)\mbox{ s.t. }\bar \eta(\Omega)\in \C{1,\phi}{\RO}, ~ \bar \eta(\Omega)>0 \mbox{ and } \exists 0<k_2<\infty \mbox{ s.t. } \Norm{\bar \eta(\Omega)}{\C{1,\phi}{\RO}}\leq k_{2} ~\forall\Omega\in{\cal O}\rbrace $.
\end{defn}

\begin{lem} \label{Lemma: Schauder estimate T}
Let equation \eqref{T} be given on a domain $ \Omega \in \O $. Suppose that $ \eta:=\bar{\eta}(\Omega) \in \C{1,\phi}{\RO}$ is the associated mapping for some $ \bar{\eta} \in \mathcal{E}^{ad}( \mathcal{O}, \phi) $ and let $ T_{e}\in \C{1,\phi}{\overline{\Oext}} $ for some $ \phi \in (0,1) $. Then there is a constant $ C>0 $ such that
\begin{equation}\label{eq: schauder inequality T} 
\Norm{T}{\C{2,\phi}{\Omega}}\leq C\left(\Norm{T}{\C{0}{\Omega}}+\Norm{\eta T_{e}}{\C{1,\phi}{\Omega}}\right) 
\end{equation}
holds for the unique solution $ T \in \C{2,\phi}{\overline{\Omega}} $. Moreover, the constant $ C $ can be chosen uniformly with respect to $ \O $. 
\end{lem}

\begin{proof}
It is easy to see that Theorem \ref{Thm: 6.30 Gilbarg/Trudinger} can be applied to 
\begin{align*}
	\left. 
		\begin{array}{ll}
			 ~~~~~\Delta T=0 &\text{ in } \Omega\\[1ex]
			 k^{-1}\eta T_{e} =k^{-1}\eta T+\sum_{1}^{3}\nu_{i}D^{i} T &\text{ on } \RO,
		 \end{array} 
	\right.
\end{align*}

\noindent for any $ \Omega \in \O $. Now we have a look into the proof of the last mentioned theorem:
 
\noindent We can choose $ l=1 $ to be the elasticity parameter. The outward normal $ \nu $ is a $ \Ck{3,\phi'}{\RO}{3} $-function and a parametrization of the boundary can be constructed using the uniformly bounded hemisphere transformations, what implicates that $ \Norm{\nu}{\C{1,\phi}{\Omega}}\leq \mathscr{L}_{1} $ independent of $ \Omega $. The $ C^{0,\phi} $-Norms of $ a_{i,j}(x)\in \{0,1\},\, b_{i}(x)=c(x)=0 $ can't surely be greater than $ \mathscr{L}_{2}=1 $ and $ \Norm{\eta}{\C{1,\phi}{\Omega}}\leq \mathscr{L}_{3} $ holds by choice of $ \eta $. Therefore, we choose $ \mathscr{L} $ as their maximum. At last, let $ \kappa < 1= \Norm{\nu}{\R{3}} $. Then we derive form inequation \eqref{eq: schauder inequality T} that
\[
\Norm{T}{\C{2,\phi}{\Omega}}\leq C\left(\Norm{T}{\C{0}{\Omega}}+\Norm{\eta T_{e}}{\C{1,\phi}{\Omega}}\right),
\]
\noindent with $ C=C_{\Omega} $ and the dependence of $ \Omega $ is subject to the choice of $ \delta $,  $ \tau $  and $ \diam(\Omega) $ appearing in the proof of \ref{Thm: 6.30 Gilbarg/Trudinger}.

\noindent Because every $ \Omega \in \O $ satisfies the uniform hemisphere condition with transformations $ \tau=\mathcal{T}^{-1} $ in Lemma \ref{Lemma: uniform hemisphere condition} we can choose $ \delta=\nicefrac{d}{2} $ in \eqref{Choice B_{0}} and as a consequence $  \delta^{\ast} = \nicefrac{d}{8}$ in \eqref{eq: 2. local inequality T}. Furthermore, $ \lambda^{\tau}_{min}(x)>0 $ depends continuously on $ \tau $ and on $ x\in \Omega $. However, we know that $ \overline{\Omega} $ is compact, and the set of transformations is as well regarding the $ C^{k,\alpha'} $-norms where $ k+\alpha'<4+\alpha $. We conclude that there has to be a lower bound $ \lambda^{*}\leq\lambda^{\tau}_{min}(x)  $ for all $ x\in \Omega $ and hemisphere transformations $ \tau $. Consequentially it is uniform w.r.t. $ \O $. The global boundedness of $ \Norm{\tau}{\Ck{1}{\Omega}{3}}\leq \Norm{\tau}{\Ck{4,\alpha}{\Omega}{3}}  $ implicates that $ \tilde{\mathscr{L}}=\mathscr{L} C_{\tau}\leq \mathscr{L}^{\ast} $ is as well, and for essentially the same reasons the assertion $ \tilde{\kappa}=\kappa \frac{C_{1}}{C_{2}} >\kappa^*$ holds, too. For $ \diam(\Omega) $ we notice that 
\[ 
 \diam(\Omega)=\sup_{x,y\in \Omega}\Norm{x-y}{\R{3}}\leq \mathcal{K}\sup_{x',y'\in \Omega_{b}} \Norm{x'-y'}{\R{3}}\leq \mathcal{K}\diam(\Omega_{b}),
\]

\noindent where $ \mathcal{K} $ is a constant depending again only on $ \tau $ and $ \Omega_{b}, $ compare \eqref{eq: (6.29) Gilbarg/Trudinger} . 
\end{proof}

The norms  of $\eta$ and $ T_{e} $ are naturally bounded by choice of these functions. What is left to be shown, is that the letter is also true for $\Norm{T}{\C{0}{\Omega}}$:

\begin{lem}\label{Lemma: max/min principle}
Let the seting of the previous Lemma be given and $ T\in \C{2,\phi}{\overline{\Omega}} $ the unique solution of \eqref{T}. For the constants $ T_{-}=\min\{T_{e}(x) \, | \, x \in \overline{\Oext} \} $ and $ T_{+}=\max\{T_{e}(x) \, | \, x \in \overline{\Oext} \}$ it holds that
\begin{equation}
	T_{-}\leq \min_{x\in \overline{\Omega}} T(x) \leq \max_{x\in \overline{\Omega}}T(x)\leq T_{+}.
\end{equation}
\end{lem}

\begin{proof} We only proof the statement for $ T_{+} $ because the proof for $ T_{-} $ proceeds analogously.

The solution $ T\in \C{2,\phi}{\overline{\Omega}} $ satisfies $\Delta T =0$ on $ \Omega $ and therefore it is harmonic as well as subharmonic. 
By application of maximum and minimum principle, see \cite{Evans,GilbTrud}, we conclude that $\min_{x\in\overline{\Omega}}T(x) =\min_{x\in \partial \Omega}T(x)$ and $\max_{x\in\overline{\Omega}}T(x)=\max_{x\in \partial \Omega}T(x).$

Suppose that there is some $ x_{0}\in \partial \Omega $ where $ T(x_{0})>T_{+} $. Then, 
\begin{equation*} 
	\frac{\partial T}{\partial \nu}(x_{0})=k^{-1}\eta(x_{0})(T_{e}(x_{0})-T(x_{0}))\leq k^{-1}\eta(x_{0})(T_{+}-T(x_{0}))<0,
\end{equation*}
because $ k^{-1}\eta >0 $ on  $ \overline{\Omega} $ and $ T_{e}(x) \leq T_{+} $ for all $ x\in \overline{\Oext} $. Let $ \varepsilon >0 $. Consistently $ x_{0}-\varepsilon\nu(x_{0})$ is contained in $ \overset{\circ}{\Omega} $ and we receive by first order Taylor series 
\begin{equation} \nonumber
	\begin{split}
		T(x_{0}-\varepsilon \nu(x_{0}))
		= T(x_{0})-\varepsilon\biggl[\nu(x_{0})^{\top}\nabla T(x_{0})-\frac{1}{\varepsilon} (r(x_{0}-\varepsilon \nu(x_{0}))\biggr],
	\end{split}
\end{equation}
where $\frac{\vert r(x_{0}-\varepsilon \nu(x_{0}))\vert}{ \Norm{\varepsilon \nu(x_{0})}{\R{3}}}=\frac{\vert r(x_{0}-\varepsilon \nu(x_{0}))\vert}{\varepsilon} \to 0 \text{ for }~\varepsilon \to 0.$ On the other hand, it holds that $ \nu(x_{0})^{\top}\nabla T(x_{0})=\frac{\partial T}{\partial \nu}(x_{0})<0.$ If we choose $ \varepsilon>0 $ small enough, the term $ \nu(x_{0})^{\top}\nabla T(x_{0})-\frac{1}{\varepsilon} (r(x_{0}-\varepsilon \nu(x_{0})) $ will be negative as well and thus $ T(x_{0}-\varepsilon \nu(x_{0}))> T(x_{0}),$ in contradiction to the maximum principle.
\end{proof}


\begin{thm}\label{Thm: uniform Schauder estimate T}
Let the setting of Lemma \ref{Lemma: Schauder estimate T} be given on a domain $ \Omega=\Phi(\Omega_{b}) \in \O $ and let $ T \in \C{2,\phi}{\Omega} $ be it's unique solution. Then, $ T\in \C{2,\phi}{\overline{\Omega}} $ and there is a positive constant $ C^{t}>0 $ such that
\begin{equation}\label{eq: uniform schauder inequality T} 
\Norm{T}{\C{2,\phi}{\Omega}}\leq C^{t},
\end{equation}
where $ C^{t} $ can be chosen uniformly with respect to $ \O $.
\end{thm}

\subsection{Schauder Estimates for the Heat Dependent Elasticity Equation}

In the last two sections we established existence of unique and uniformly bounded solutions for elasticity equation and heat equations separately. In order to derive results for the combined problem we tie up to inequality \eqref{eq: estimate u_{T}}: There, we assumed $ T\in \C{2,\phi}{\overline{\Oext}} $, but in fact we only have $ T\in \C{2,\phi}{\overline{\Omega}} $. \\
Moreover, extensions of functions will be needed to define convergence of solution sequences in the sense of Section \ref{Sec: basic notations} and to poof compactness of the appropriate graph, see Section \ref{Sec: exsistence of optimal shapes}.

\begin{defn}[\bf{ State Problem and State Space for Thermo Elasticity}]\label{Def: state problem}$  $ 
\begin{itemize}
	\vspace*{-0.3cm}
\item[(i)] Let $\Omega:=\Phi(\Omega_{b}) $ be an $ C^{4,\alpha} $-admissible shape for some $ \alpha \in (0,1)$. Moreover, decompose the boundary into  the interior boundary $\partial\Omega_{D}:=\partial \Phi(B) $ and the complete exterior boundary $ \partial\Omega_{N}=\partial\Omega \setminus \partial \Omega_{D} $. Then, the according state problem for thermo elasticity $ \mathcal{P}(\Omega) $ is given by equation \eqref{PTO}. 
 
\item[(ii)] Let $0<\varphi<\phi<1 $. Moreover, assume that $ k>0 $ and that the Lamé coefficients $ \lambda,\, \mu>0 $ are constants, as well as $ \rho $. Additionally, choose $ \eta:=\bar{\eta}(\Omega)\in \mathcal{E}$, $ T_{e}\in \C{1,\phi}{\overline{\Oext}} $, $ f \in \Ck{1,\phi}{\overline{\Oext}}{3},\,g:=\bar{g}(\Omega) \in G$ and $ 0<\varphi'<\varphi $, $ 0<\phi'<\phi $. Then, the state space for thermo elasticity is given  by $ V_{P}(\Omega)= \C{2,\phi'}{\overline{\Omega}}\oplus\Ck{3,\varphi'}{\overline{\Omega}}{3} $.
\end{itemize}
\end{defn}

\noindent Note that the chosen decomposition of the boundary depends continuously on the choice of $ \Phi \in \Uad $. Moreover the two dimensional Lebesgue-measures of $ \partial\Omega_{N} $ and $ \partial\Omega_{D} $ are always bounded away from zero.

\begin{thm}\label{Thm: uniform Schauder estimate u_{T}}
The state Problem $\mathcal{P}(\Omega)$ be defined as in Definition \ref{Def: state problem}. Then for every $ \varphi \in (0,\phi) $ there exists a unique solution $ (T,u_{T})\in V_{P}(\Omega) $ of \eqref{PTO} where
\begin{equation}
\begin{split}
\Norm{T}{\C{2,\phi}{\Omega}}&\leq C^{t}\\
\Norm{u_{T}}{\Ck{3,\varphi}{\Omega}{3}}&\leq C^{t/el}
\end{split}
\end{equation}
holds for the constant $ C^{t}>0 $ of Theorem \ref{Thm: uniform Schauder estimate T} and some constant $C^{t/el}>0$ that also can be chosen independent of $ \Omega $. Remark that the tuple $ (T,u_{T}) $ actually is an element of $ \C{2,\phi}{\overline{\Omega}}\oplus\Ck{3,\varphi}{\overline{\Omega}}{3} $.
\end{thm}

\begin{proof}
The existence of the unique solutions $ T\in \C{2,\phi}{\overline{\Omega}} $ and $ u \in \Ck{3,\varphi}{\overline{\Omega}}{3} $ follows directly from Theorem 6.31 in \cite{GilbTrud} and Theorem \ref{Thm: schauder estimates u}, confer Theorem 5.6 and 5.7 in \cite{Hanno}, too.
By \eqref{eq: estimate u_{T}} we conclude that for $ \varepsilon>0 $ small enough such that $\varepsilon C<1$ 
\begin{equation*}
\begin{split}
\Norm{u_{T}}{\Ck{3,\phi}{\Omega}{3}}\leq& \tfrac{C_{f,g}}{1-\varepsilon C}+ \tfrac{C\rho(3\lambda+2 \mu)}{1-\varepsilon C}\left(\Norm{\nabla T}{\Ck{1,\phi}{\Omega}{3}} +(|T_{0}|+\Norm{T}{\C{2,\phi}{\RO}})\tilde{C}\right)
\leq  C^{\ast},
\end{split}
\end{equation*}
since $C$ was chosen uniformly, see Lemma \ref{Lemma: Schauder estimate T}. The constant $ C_{f,g} $ can be chosen to be the same for any $ \Omega \in \O $ due to the choice of $ f $ and $ g\in \mathcal{E} $. The letter is also true for $ \tilde{C} $ as already mentioned in the beginning of Section \ref{Subsec: solutions for heat equation}. Additionally, $ \Norm{\nabla T}{\Ck{1,\phi}{\Omega}{3}}\leq c\Norm{(\nabla T)_{i}}{\C{1,\phi}{\Omega}}\leq \Norm{T}{\C{2,\phi}{\Omega}} \leq C^{t}$ uniformly, because the constant $ c $ results from the equivalence of Norms on $ \R{3} $ and is independent of $ \Omega $. The starting temperature $ T_{0} $ is constant. Therefore, $ C^{\ast}$ can be chosen to be uniform w.r.t. $ \O $. 
\end{proof}



\section{\label{sec:OpShapes}Existence of Optimal Shapes}\label{Sec: exsistence of optimal shapes}

In this section we prove existence of optimal solutions to shape optimization problems where the constraints are given by thermal elasticity, see Section \ref{Sec: PDE constraints}, and the cost functionals are of very general class. Since they are to singular to be defined on base of weak solutions we have to resort to regularity theory and strong solutions: These functionals include surface integrals which lead to a loss of regularity according to the appearing derivatives of $ u $ and $ T $ and the trace theorem\footnote{Confer for example \cite[5.5]{Evans}}. 

\begin{notation}
The objective is to find an optimal shape $ \Omega=\Phi(\Omega_{b}) $ within the set of $ C^{4,\alpha} $-admissible shapes $ \O  $, defined in \ref{Def: k-adm domains}, which minimizes a local cost functional $ \mathcal{J}(\Omega,u,T) =\mathcal{J}_{vol}(\Omega,u,T) +\mathcal{J}_{sur}(\Omega,u,T) $, where
\begin{equation}\label{Def: local cost functional}
\begin{split}
	\mathcal{J}_{vol}(\Omega,u,T)&=\int_{\Omega}\mathcal{F}_{vol}(x,T,\nabla T,\nabla^2 T, u,\nabla u,\nabla^{2}u,\nabla^{3}u)dx\\[1ex]
	\mathcal{J}_{sur}(\Omega,u,T)&=\int_{\RO}\mathcal{F}_{sur}(x,T,\nabla T,\nabla^2 T,u,\nabla u,\nabla^{2}u,\nabla^{3}u)dA.
\end{split}
\end{equation}
Here, the tuple $ (T,u) $ solves the state Problem $ \mathcal{P}(\Omega) $.
\end{notation}

In Sections \ref{Sec: form_design} to \ref{Sec: uniform schauder estimates} we prepared the application of Theorem \ref{Thm: existence of optimal shapes} in the present constellation. Now, we are able to proof that the graph $ \mathcal{G}=\{(\Omega,T,u) \, \vert \, \Omega\in \O \} $ is compact in the sense of \eqref{Def: Gcomp}. In the following definition we therefor choose $ q=2,\, \beta=\phi' $ respectively $q=3,\, \beta=\varphi' $.

\begin{defn}[\textbf{$ C^{q,\beta} $ Convergence of functions on Variable Domains}]\label{Def: Conv. on var. domains} $  $ \\
Recall the sets $ \O $ and $ \Oext $ of Section \ref{Sec: form_design}. Let $ m,q \,\in \N{} $, $ \beta \in (0,1) $ be fixed.
\begin{itemize}
\item[\textit{i)}] Let
$ 
p^{m,q,\beta}_{\Omega}:\big[\C{q,\beta}{\overline{\Omega}}\big]^{m} \to \big[C^{q,\beta}_{0}(\Oext) \big]^{m} 
$ 
be the extension operator that exists by Lemma \ref{Lemma: extension Lemma 2}. If $ v \in \big[ \C{q,\beta}{\overline{\Omega}}\big]^{m} $ set $ v^{ext} =p^{m,q,\beta}_{\Omega} v$. 
\item[\textit{ii)}]
Let $ \seq{\Phi_{n}} \subset \Uad,\, \Phi \in \Uad $ and $ \Omega_{n}:=\Phi_{n}(\Omega_{b})\in \O,\, \Omega=\Phi(\Omega_{b})\in \O $. For $ \seq{u_{n}} $ with $ u_{n} \in \big[\C{q,\beta}{\overline{\Omega_{n}}}\big]^{m},\, n\in \N{} $ and $ u\in \big[\C{q,\beta}{\overline{\Omega}}\big]^{m} $ we define the expression $ u_{n} \rightsquigarrow u $ as $ n\to \infty $ by $ u_{n}^{ext}\to u^{ext} $ in $ \big[C^{q,\beta}_{0}(\Oext) \big]^{m} $.
\end{itemize}
\end{defn}

\begin{lem}
Let $ \Omega \in \O $ be a $ C^{4,\alpha} $ admissible shape for some $ \alpha \in(0,1) $ and suppose $ v\in \C{q,\beta}{\overline{\Omega}} $, where $1\leq q+\beta \leq 4+\alpha $. Then there exists a function $ w\in C^{q,\beta}_{0}(\Oext) $ and a constant $ C_{q}>0 $ such that $ w=v $ in $ \Oext $ and 
\[ 
\Norm{w}{\C{q,\beta}{\Oext}}\leq C \Norm{v}{\C{q,\beta}{\Omega}}
\]
where $ C=C_{q} $ is independent of $ \Omega $ and $ \Oext $.
\end{lem}

\begin{proof}
Confer the proof of Lemma \ref{Lemma: extension Lemma 1} (see \cite{GilbTrud}) and substitute $ \Omega' $ by $ \Oext $ and $ \Omega $ by $ \Omega_{b} $. Furthermore, we replace the $ C^{k,\varphi} $ diffeomorphism $ \psi $ by the hemisphere transformations $ \mathcal{T}_{x_{0}}: \Umg{d}{x_{0}} \to \Sigma_{R} $, $ x_{0}\in \partial \Omega_{b} $ introduced in Lemma \ref{Lemma: uniform hemisphere condition}, \eqref{eq: hemisphere Transformations Omega_b }. Then, $ G^{+}=\Sigma_{R} $ and $ G=\Umg{R}{0}\supset \Sigma_{R} $ is a ball with radius $ R $ at the origin of ordinates. For $ u\in \C{q,\beta}{\overline{\Omega_{b}}} $, $ q+\beta\leq 4+\alpha $ one sets $ \tilde{u}(y)=u \circ \mathcal{T}_{x_{0}}(y) $, where $ y=(y',y_{3}) $, $ y'=(y_{1},y_{2}) $. We follow \cite{GilbTrud} and define an extension into $ y_{3}<0 $ by
\[
\tilde{u}(y',y_{3})=\sum_{i=1}^{q}c_{i}\tilde{u}(y',-y_{3}/i),~, y_{3}<0,~~~
\sum_{i=1}^{q}c_{i}\left(\nicefrac{-1}{i}\right)^{m},~~m=0,\ldots, q+1.
\]
Furthermore, the consulted proof shows that $ w=\tilde{u}\circ \mathcal{T}_{x_{0}}^{-1} $ provides a $ C^{4,\alpha} $-extension of $ u $ onto $ \Omega_{b}\cup B(x_{0}) $ for some balls $ B(x_{0}) $ and the related hemisphere transformation. A finite covering argument of $ \partial \Omega_{b} $ and an associated partition of unity leads to the sought extension $ w\in C^{q,\beta}_{0}{\Oext} $. The fact, that the constant $ C $ appearing in the inequality

\begin{equation}\label{eq: norm of extension}
\Norm{w}{\C{q,\beta}{\Omega'}}\leq C\Norm{u}{\C{q,\beta}{\Omega}}
\end{equation}  

\noindent only depends on $ q,\Omega',\Omega $ is owed to inequalities (6.29) and (6.30) in \cite{GilbTrud}, confer the proof of Lemma 6.37. So, in our case, it depends mainly on $ \Oext $ and $ \Omega_{b} $, which are both fixed sets.

Given a function $ v\in \C{q,\beta,}{\Omega} $, $ 1\leq q+\beta\leq 4+\alpha $ on $ \Omega=\Phi(\Omega_{b}) $ for some $ \Phi \in \Uad $ the mapping $ u=v \circ \Phi $ defines a $C^{q,\beta}$ function on $ \Omega_{b} $. Then, an extension $ w\in \C{q,\beta}{\Oext} $ can be defined as described above. Hence, $ w=v \circ \Phi $ on $ \Omega_{b} $ and $ w\circ \Phi^{-1}=v $ on $ \Omega $ is an extension of $ v $, since $ \Phi $ is a $ C^{4,\alpha} $ diffoemorphism on $ \Oext $. By application of (6.30) \cite{GilbTrud} and \eqref{eq: norm of extension} we receive
\[ \Norm{w\circ \Phi^{-1}}{\C{q,\beta}{\Oext}}\leq \mathcal{C} \Norm{w}{\C{q,\beta}{\Oext}}\leq \mathcal{C}C\Norm{v \circ \Phi}{\C{q,\beta}{\Omega_{b}}}\leq \mathcal{C}^{2}C\Norm{v}{\C{q,\beta}{\Omega}}, \]
for a suitable positive constant $\mathcal{C} $ that can be chosen uniformly w.r.t $ \Uad $ due to it's construction.
\end{proof}

\begin{lem}[\textit{\textbf{Compactness of the Graph}}]\label{Lemma: compactness of the graph}
Let $ \seq{\Omega_{n}}=\seq{\Phi_{n}(\Omega_{b})} \subset \O $ be an arbitrary sequence, where on any $ \Omega_{n} $ the setting of Theorem \ref{Thm: uniform Schauder estimate u_{T}} is given. By $ (T_{n},u_{n})\in V_{P}(\Omega_{n}) $ we denote the corresponding solutions to the state problem $\mathcal{P}(\Omega_{n}) $. Then, the sequence $ \seq{\Omega_{n},T_{n},u_{n}} $ has a subsequence $\subseq{\Omega_{n_{k}},T_{n_{k}},u_{n_{k}}}$ such that $ \Omega_{n} \xrightarrow{\O} \Omega $, $ \Omega=\Phi(\Omega_{b}) $ as $ k\to \infty $, as well as $ T_{n_{k}}\underset{k\to \infty}{\rightsquigarrow} T $ and $ u_{n_{k}}\underset{k\to \infty}{\rightsquigarrow} u $ for the corresponding solutions $ T $ and $ u $ to $ \mathcal{P}(\Omega) $, where $ (T,u)\in V_{P}(\Omega) $.
\end{lem}

\begin{proof}
In terms of Lemma \ref{Lemma: compactnes Uad} there is a convergent subsequence $ (\Phi_{n_{l}})_{l\in \N{}}$ tending to some $ \Phi \in \C{4,\alpha}{\overline{\Oext}} $ as $ l\to \infty $ concerning $ \Norm{.}{\Ck{k,\alpha'}{\Oext}{3}} $, $ 0\leq\alpha'<\alpha $. Hence, $\Omega_{n}\xrightarrow{\O}\Omega:=\Phi(\Omega_{b}) $, when passing to the limit. According to Theorem \ref{Thm: uniform Schauder estimate T} $ T_{n_{l}}\in \C{2,\phi}{\overline{\Omega_{n_{l}}}} $ for every $ n_{l} \in \N{} $. Moreover, this theorem leads to $ \Norm{T_{n_{l}}}{\C{2,\phi}{\overline{\Omega_{n_{l}}}}}\leq C^{t} $ for every $ n_{l} $. Let $ p_{\Omega_{n_{l}}}:=p_{\Omega_{n_{l}}}^{1,2,\phi} $ be the extension operator in Definition \ref{Def: Conv. on var. domains} and set $ T_{n_{l}}^{ext}:=p_{\Omega_{n_{l}}}T_{n_{l}} $. Because of Lemma \ref{Lemma: extension Lemma 1} there is a constant $ C^{ext} $ such that $ \Norm{T_{n_{l}}^{ext}}{\C{2,\phi}{\Oext}}\leq C^{ext}\Norm{T_{n_{l}}}{\C{2,\phi}{\Omega_{n_{l}}}} \leq C^{ext}C^{t}$. The constant $ C^{ext} $ can again be chosen uniformly with regard to $ \O $, what was shown in the previous lemma. This results in a uniform bound for all $ \Norm{T_{n_{l}}^{ext}}{\C{2,\phi}{\Oext}} $. By means of Lemma \ref{Lemma: compacness of c_k_alpha}, there is again a convergent subsequence $ T^{ext}_{n_{j}}\to T^{ext,*}$ as $ j \to \infty $ concerning $ \Norm{.}{\C{2,\phi'}{\Oext}} $. The limit $ T^{ext,*}\in \C{2,\phi'}{\overline{\Oext}} $ even is an element of $ \C{2,\phi}{\overline{\Omega}} $, see the proof of Lemma \ref{Lemma: compactnes Uad}. Thus, 
\[ 
\Omega_{n_{j}}\xrightarrow{\O}\Omega ,\, T_{n_{j}}^{ext}\xrightarrow{C^{2,\phi'}}T^{ext,*} ~~ \text{for some } T^{ext,*} \in \C{2,\phi}{\overline{\Omega}}.
\]
Because, in particular, all derivatives of $ T_{n_{j}}^{ext} $ up to order two converge to those of $ T^{ext,*} $ as $ j\to \infty $ the function $ T^{ext,*}\restriction_{\Omega} $ solves \eqref{T}. For Theorem \ref{Thm: uniform Schauder estimate T} holds, \eqref{T} has a unique solution $ T $, and $ T^{ext,*} $ has to be an extension of $ T $ to $ \overline{\Oext} $.

We apply the same arguments to the solutions $ u_{n_{j}}\in \Ck{3,\varphi}{\Omega_{n_{j}}}{3} $ that are allocated to $ \Phi_{n_{j}}$ and $T^{ext}_{n_{j}} $. In this way, we finally obtain a further subsequence $(\Omega_{n_{k}},T^{ext}_{n_{k}},u^{ext}_{n_{k}})_{k\in \N{}}  $, where $ u^{ext}_{n_{k}}=p_{\Omega_{n_{k}}}^{3,3,\varphi}u_{n_{k}} $ and
\begin{equation*}
		\Omega_{n_{k}} \overset{\O}{\longrightarrow}\Omega,~~
		T^{ext}_{n_{k}}\xrightarrow{C^{2,\phi'}} T^{ext,*},~~ 
		u^{ext}_{n_{k}} \xrightarrow{C^{3,\varphi'}} u^{ext,*} 
		~~ \text{for some } u^{ext,*} \in \C{3,\varphi}{\overline{\Omega}}.
\end{equation*}
\noindent Actually $ (T^{ext,*}, u^{ext,*})\restriction_{\Omega} $ solves \eqref{PTO}, because $T^{ext,*}\restriction_{\Omega}$ solves \eqref{T} and $ u^{ext}_{u_{l}}\xrightarrow{C^{3}} u^{ext,*} $, where $ (T^{ext,*}, u^{ext,*}) $ is an extension of the unique solution $ (T,u) $ of \eqref{PTO}.
\end{proof}

\begin{lem}[\textbf{\textit{Continuity of Local Cost Funktionals}}]
\label{Lemma:Continuity}
Let $ \mathcal{F}_{vol},\, \mathcal{F}_{sur} \in C^{0}(\R{d})$ with $ d=3+\sum_{j=0}^{s}3^j+\sum_{j=0}^{r}3^{j+1},\, s=2,r=3 $, and let the set $ \O $ only consist of $ C^{0} $-admissible shapes. For $ \Omega \in \O $, $ T\in\C{2}{\overline{\Omega}} $ and $ u\in \C{3}{\overline{\Omega}} $ consider the volume integral $  J_{vol}(\Omega,T,u) $ and the surface integral $ J_{sur}(\Omega,T,u) $ of \eqref{Def: local cost functional}.\\[1ex]
Let $ \seq{\Omega_{n}}\in\O $ with $ \Omega_{n} \xrightarrow{\O}\Omega $ as $ n \to \infty $, $ \seq{u_{n}}\in \Ck{3}{\overline{\Omega_{n}}}{3} $ be a sequence with $ u_{n}\rightsquigarrow u $ $ (q=3,\, \beta=0) $ and $ \seq{T_{n}}\in \C{2}{\overline{\Omega_{n}}} $ be a sequence with $ T_{n}\rightsquigarrow T $ $ (q=2,\, \beta=0) $. Then,

\begin{itemize}
\item[(i)] $\mathcal{J}_{vol}(\Omega_{n},T_{n},u_{n})\to J_{vol}(\Omega,T,u)$ as $ n\to \infty $.
\item[(ii)] If the family $ \O $ consists only of $C^{1} $-admissible shapes, then $\mathcal{J}_{sur}(\Omega_{n},T_{n},u_{n})\to J_{sur}(\Omega,T,u)$ as $ n\to \infty $.
\end{itemize}
\end{lem}

\begin{proof}
(i) Statement $ (i) $ can be proofed in the same way as Lemma 6.3 in \cite{Hanno}.\\[1ex]
(ii) Because of its definition, every $ \Omega \in \O $ has a boundary that can be considered as a differentiable  2-dimensional submanifold in $ \R{3} $:
Let $ x_{0}=\Phi(z_{0}),\, z_{0}\in \partial\Omega_{b} $ be some point in the boundary. The mapping $  \T_{x_{0}}:B_{d}(x_{0})\cap \RO \to F_{R}\subset \R{2}\times \{0\}$ defines a chart for $ x_{0} $ if $ \T_{x_{0}},\,d $ and $ F_{R} $ are defined analogously to Lemma \ref{Lemma: uniform hemisphere condition}. Now we can choose $x_{i},i=1\ldots,l \in \N{} \, \in \RO$, such that $ \RO \subset \bigcup_{i=1}^{l} B_{d}(x_{i})\cap \RO $. By restriction of these mappings to carefully chosen sets $ \A^{i}\subset B_{d}(x_{0})\cap \RO  $, we can define an atlas $ (\A^{i})_{i=1,\ldots,l}  $ for $ \RO $, where $ \A^{i}\cap \A^{j}=\emptyset $. In this way, we can write the surface integral as a sum:

\begin{equation*}
\begin{split}
J_{sur}(\Omega_{n},T_{n},u_{n})
&=\int_{\RO_{n}}\mathcal{F}_{sur}(x,T_{n},\nabla T_{n},\nabla^2 T_{n},u,\nabla u,\nabla^{2}u,\nabla^{3}u)\, dA\\
&=\sum_{i=1}^{l}\int_{\A_{i}^{n}}\mathcal{F}_{sur}(x,T_{n},\nabla T_{n},\nabla^2 T_{n},u,\nabla u,\nabla^{2}u,\nabla^{3}u)\, dA\\
\end{split}
\end{equation*}

\noindent If we denote the chart-mappings by $h_{n}^{i}: \A^{i}_{n} \to \tilde{\A}^{i}_{n} $ and the correspinding Gram determinants by $ g^{h^{i}_{n}} $, we can write the integrals in the form
\[ 
\int_{\A_{i}^{n}}\mathcal{F}_{sur}\big(h^{i}_{n}(s),T_{n}(h^{i}_{n}(s)),\nabla T_{n}(h^{i}_{n}(s)),\nabla^2 T_{n}(h^{i}_{n}(s)),u(h^{i}_{n}(s)), 
\ldots,\nabla^{3}u(h^{i}_{n}(s))\big)\sqrt{g^{h^{i}_{n}}(s)} \, ds.
\]
Especially, the $ h_{n}^{i} $ corresponds to the inverse hemisphere transformations on $ \Omega_{n} $. Since $ \Phi\in C^{1}(\overline{\Oext})\leq K $ by some constant $ K $, which is the same for all $ \Omega \in \O $ and the hemisphere transformations on $ \Omega_{b} $ are uniformly bounded and the Gram determinant is a well. Because of $ \mathcal{F}_{sur}\in C^{0}(\R{d}) $, $ T_{n}\rightsquigarrow T $ as $ n\to \infty $ and $ u_{n}\rightsquigarrow u $ as $ n\to \infty $, there are constants $ C_{i} $ such that 
\[ 
\left| \mathcal{F}_{sur} \big(h^{i}_{n}(s),T_{n}(h^{i}_{n}(s)), \nabla T_{n}(h^{i}_{n}(s)),\nabla^2 T_{n}(h^{i}_{n}(s)),u(h^{i}_{n}(s)),\ldots, \nabla^{3}u(h^{i}_{n}(s))\big) \sqrt{g^{h^{i}_{n}}(s)}  \right| \leq C_{i}
\]
holds for all $ n\in \N{},\, i=1,\ldots,l $. In consequence, Lebesgue's theoerem of dominated convergence can be applied (with $ C=\max{C_{i}: i=1,\ldots,l} $) and we conclude $ J_{sur}(\Omega_{n},T_{n},u_{n}) \to J_{sur}(\Omega,T,u) $ when passing to the limit. 
\end{proof}

\begin{thm}[\textbf{\textit{Solution to the SO Problem}}]
\label{thm:SolSO}
Let the set of admissible shapes be given in Def. \ref{Def: k-adm domains} with $k=4$. Then the shape optimization problem \eqref{PO} with the objective functionals \eqref{Def: local cost functional} and the thermomechanical state equation \eqref{PTO} has at least one solution $\Omega^*\in\mathcal{O}$.
\end{thm}
\begin{proof}
As demonstrated in the above Lemmas \ref{Lemma: compactness of the graph} and \ref{Lemma:Continuity}, the conditions of Theorem \ref{Thm: existence of optimal shapes} are fulfilled and the assertion follows.
\end{proof}

The application of this result in shape optimization to the optimal reliability problem now is straight forward:

\begin{cor}[\textbf{\textit{Solution to the Optimal Reliability Problem}}]
For all optimal reliability problems from Definitions  \ref{def:Reliability} and \ref{def:OptReliability} there exists at least one solution in the set of admissible shapes ${\cal O}$ for the local, probabilistic failure time model for LCF.  
\end{cor}
\begin{proof}
Combine the results of Lemma \ref{lem:OpRelSO} and \ref{lem:Model} with the above Theorem \ref{thm:SolSO}.
\end{proof}
\vspace{.5cm}

We note that these results hold analogously, if the set of admissible shapes with volume constraints is considered, cf. Remark \ref{rem:VolumeConstraints}.

\section{\label{sec:SO} Summary and Outlook}

In the present paper we have shown the existence of optimal solutions to a class of shape optimization problems with the thermo-mechanic PDE as the state equation. The objective functionals can be rather singular type, which forces us to use elliptic regularity theory and domains defined by smooth deformation of a baseline domain. We have also shown how this relates to the notion of optimal reliability of a mechanical design. This generalizes prior work \cite{Hanno,Schmitz} in several respects: A more general setting for optimal reliability problems, a temperature dependet crack initiation process, more flexible admissible shapes and a multi-physical state equation.

A number of new questions naturally arise at his point: The first concerns the construction of boundary value problems that associate surface forces to shapes $\Omega$. This point has been left open in this article and was not even mentioned in previous work \cite{Hanno,Schmitz}. From an applied prospective, such models should come from other physical processes, such as static gas pressure $g=P\nu$ on $\partial\Omega$, that will again depend on the geometry. Taking a potential flow $v$ in a region exterior to $\Omega$ as a simple model, it should be possible to verify the assumptions of Definition \ref{Def: SurfaceForceModel} for $P=P(v)$ using elliptic regularity theory once again. A quick check however reveals that the $2$nd order Shauder estimates for the Poisson equation that are proven in this work are one order too low to meet the $C^{2,\phi}$ continuity requirements for $g$ in the mechanical Schauder estimate. One thus has to use the more general framework of \cite{Agm64} in order to treat even the simplest physical boundary condition model in the framework of elliptic regularity theory. A even more multi-physical approach, starting with simple flow models in the exterior of $\Omega$ and proceeding to more complicated ones, seems to be an interesting research direction for the future.

It would also be desirable, to study the continuous shape derivative of failure probabilities in the given context, see \cite{SokZol92} for the general theory and \cite{BGS,Schmitz} for some first steps in that direction. The mathematically rigorous treatment of shape derivatives for rather singular objective functionals is not an easy task. Material and shape derivatives should have a similar $C^{k,\phi}$ regularity class as the solutions, but also depend on the solutions \cite{SokZol92}, so a careful treatment is desirable, in particular if one would like to consider higher derivatives. Furthermore, a formal inspection of the right hand side of the adjoint equation, for the class of objective functionals given by the optimal reliability application, reveals that the adjoint state can not be a Sobolev function, since the formal expressions for $\frac{\partial J_{sur}(\Omega,T,u)}{\partial u}$ do not define a functional in $H^{-1}(\Omega)$. This raises several interesting questions on the nature and numerical approximation of the adjoint state that are beyond the scope of the present article. 

\vspace{0.5cm}

\noindent \textbf{Acknowledgements:} We would like to thank Tilman Beck,  Rolf Krause, Nadine Moch  Georg Rollmann, Mohamed Saadi, Sebastian Schmitz and Thomas Seibel for interesting discussions. We also thank T. Beck for the permission to reproduce Figure \ref{fig:Damage} (b).

\appendix 
\section{Appendix}

\begin{lem} \label{Lemma: compacness of c_k_alpha} \cite[Lemma 6.36]{GilbTrud} 
Let $ \Omega $ be a $ C^{k,\phi} $-domain in $ \R{n} $ (with $ k \geq 1 $) and let $ S $ be a bounded set in $ \C{k,\phi}{\overline{\Omega}} $. Then $ S $ is precompact in $ \C{j,\beta}{\overline{\Omega}} $ if $ j+\beta < k+\phi $.
\end{lem}

\begin{defn}[\textbf{Hemisphere Property}]\label{Def: hemisphere transformation} \cite[S. 667]{Agm64}$ $\\
Let $ \Omega \subset \R{n} $ be a domain with a  $ C^{k,\phi} $ boundary portion $ \Gamma $ and $ A \subseteq \Omega $  be a subdomain such that $ \partial A \cap \partial \Omega \subseteq \overset{\circ}{\Gamma} $ in the $ (n-1) $-dimensional sense.\\
$ A $ is said to satisfy a $ C^{k,\phi} $-hemisphere property on $ \Gamma $, if there exists constant $ d>0 $ such that every $ x \in A$ with $ \dist(x,\Gamma)\leq d $ possesses a neighborhood $ U_{x}\subset \R{n} $ where 
\begin{itemize}
\item[(i)]$ \overline{U}_{x} \cap \partial \Omega \subseteq \Gamma$,
\item[(ii)]$ \Umg{d/2}{x} \subseteq U_{x} $ and
\item[(iii)](a) $\overline{U}_{x} \cap \overline{\Omega}=\T(\Sigma_{R(x)}),  ~~~~~~~$(b)$ \overline{U}_{x} \cap \partial \Omega=\T(F_{R(x)}) $, $ 0<R(x)\leq 1 $
\end{itemize}
for some hemisphere $ \Sigma_{R(x)} $ and it's flat boundary $ F_{R(x)} $.
The transformations $ \T,\T^{-1} \in C^{k,\varphi} $ are dependent on the point $ x\in A $.
\end{defn}

\begin{lem}\label{Lemma: extension Lemma 1}\cite[Lemma 6.37]{GilbTrud}\\
Let $ \Omega $ be a $ C^{k,\phi} $ domain in $ \R{n} $ (with $ k\geq 1 $) and let $ \Omega' $ be an open set containing $ \overline{\Omega} $. Suppose $ u\in \C{k,\phi}{\overline{\Omega}} $. Then there exists a function $ w\in C^{k,\phi}_{0}(\Omega') $ such that $ w=u $ in $ \Omega' $ and 
\[ 
\Norm{w}{\C{k,\phi}{\Omega'}}\leq C \Norm{u}{\C{k,\phi}{\Omega}},~~~C=C_{k,\Omega,\Omega'}
\]
\end{lem}

\begin{lem}\label{Lemma: extension Lemma 2}\cite[Lemma 6.38]{GilbTrud} \\
Let $ \Omega $ be a $ C^{k,\phi} $ domain in $ \R{n} $ (with $ k\geq 1 $) and let $ \Omega' $ be an open set containing $ \overline{\Omega} $. Suppose $ \psi\in \C{k,\phi}{\RO} $ or $ \psi \in \C{k,\phi}{\overline{\Omega}} $\footnote{See the explanations concerning $ C^{k,\phi} $ domains at the beginning of Chapter 6.2 in \cite{GilbTrud}}. Then there exists a function $ \Psi \in C^{k,\phi}_{0}(\Omega') $ such that $ \Psi=\psi $ on $ \RO $, $ \overline{\Omega} $ respectively.
\end{lem}


\end{document}